\documentclass[
    11pt,
    letterpaper,
    reprint,
    notitlepage,
    superscriptaddress,
    aps,prx,
]{revtex4-2}

\usepackage{amsmath,amssymb,amsfonts} 

\usepackage{amsthm} 
\newtheorem{lemma}{Lemma}[section]
\newtheorem{theorem}{Theorem}
\newtheorem{definition}{Definition}[section]

\usepackage{cmupint} 
\usepackage{physics}

\usepackage{enumitem}
\setlist{nolistsep}

\usepackage{microtype}

\usepackage[svgnames,dvipsnames]{xcolor}
\usepackage{graphicx}
\definecolor{NewBlue}{rgb}{0.1, 0.1, 0.7}
\definecolor{NewRed}{rgb}{0.7, 0.1, 0.1}
\usepackage[colorlinks,
    linkcolor=NewBlue,
    citecolor=NewBlue,
    urlcolor=NewBlue]{hyperref}

\usepackage[capitalise]{cleveref}
\newcommand{\newsec}[1]{\textbf{#1}}

\renewcommand{\t}[1]{\mathrm{{#1}}}
\renewcommand{\d}{\dd}

\newcommand{\eps}{\varepsilon}

\renewcommand{\vec}[1]{\vb*{#1}}

\newcommand{\ii}{\mathrm{i}}
\newcommand{\cc}{\star}
\newcommand{\tg}{\mathrm{tan}}

\newcommand{\p}{\partial}

\usepackage{empheq}
\newcommand{\eqdef}{\coloneqq}

\DeclareMathOperator*{\arginf}{arg\,inf}

\newcommand{\LigoMIT}{LIGO Laboratory, Massachusetts Institute of Technology, Cambridge, MA 02139}
\newcommand{\MechMIT}{Department of Mechanical Engineering, Massachusetts Institute of Technology, Cambridge, MA 02139}
\newcommand{\PhysHarv}{Department of Physics, Harvard University, Cambridge, MA 02139}
\newcommand{\TrinCollCam}{Trinity College, University of Cambridge, United Kingdom}

\begin{document}

\title{Forecasting in the presence of scale-free noise}
\author{Serhii Kryhin}
\affiliation{\PhysHarv}
\email{skryhin@g.harvard.edu}
\author{Tatiana Mouzykantskii}
\affiliation{\TrinCollCam}
\date{\today}
\author{Vivishek Sudhir}
\affiliation{\LigoMIT}
\affiliation{\MechMIT}
\email{vivishek@mit.edu}

\begin{abstract}
The extraction of signals from noise is a common problem in all areas of science and 
engineering \cite{Kolm39,Wien49,Kalm60,KailHass00}. A particularly useful version
is that of forecasting: determining a causal filter that estimates a future value of a 
hidden process from past observations. Current techniques for deriving the filter require that the
noise be well described by rational power spectra. However, scale-free noises, whose spectra
scale as a non-integer power of frequency, are ubiquitous in practice. 
We establish a method, together
with performance guarantees, that solves the forecasting problem in the presence of
scale-free noise. Via the duality between estimation and control, our technique can be used to design 
control for distributed systems. These results will have wide-ranging applications in neuroscience, finance, fluid
dynamics, and quantum measurements.
\end{abstract}

\maketitle

\newsec{Introduction.}
The problem of forecasting signals from their noisy measurements was first 
formulated by Kolmogorov and Wiener \cite{Kolm39,Wien49}.
The solution, via the so-called Wiener-Hopf factorization, involves determining an analytic function in the complex plane from its boundary values, a procedure that can be performed when the noise is characterized
by a rational function of frequency \cite{WienHop31,Krein58,SayKail01,KisRog21}.
A state-space reformulation, the Kalman filter \cite{Kalm60,KailHass00}, 
is a mainstay of modern forecasting. 
Both formulations, especially when performance guarantees are required, are limited to observation 
processes with rational power spectra, an ubiquitous 
assumption in the modern theory of time-series analysis \cite{Brill01,Ham94}.

Equally ubiquitous is the practical relevance of forecasting in the presence of
noises which have no rational, or even analytic, spectra.
Examples abound: velocity fluctuations in
turbulent flow \cite{Kolm41turb,Batch59,Pope00,BruCarb13}, height fluctuations in ocean 
waves \cite{Ryab19,Holth10} and rivers \cite{MandWall69}, 
fluctuations in classical \cite{Calo74,Ziel79} and quantum electronics \cite{PalAlt14,BrowBla15,Mull19,YonTar23},
displacement fluctuations of mechanical oscillators used to test quantum mechanics and 
gravity \cite{GroAsp15,FedKip18,CripCorb19},
action potential fluctuations in the neural system \cite{Prit92,Mill09,neuro_noise}, 
fluctuations in financial markets \cite{Mand63,ManStan95,GhaDod96} and economies \cite{Gab09}, 
and even the frequency of human conflicts \cite{Rich60,Clau18}.
In all such examples, the relevant processes have spectra that are non-rational and/or 
scale with frequency $\omega$ as $\omega^{-\alpha}$ as $\omega \rightarrow 0$ for 
some (possibly non-integer) index $\alpha > 0$. State-space models driven by white noise cannot generate 
such noises, obviating the
Kalman filter paradigm. The Wiener filter paradigm could still hold, but it is unclear how to
formulate and deploy it.

We formulate the causal Wiener filter as a constrained variational problem, which results in
a pair of operator equations whenever the spectrum of the noise is square-integrable. We identify a symmetry
in the variational formulation that allows the filtering problem for a much larger class of noises to be transformed
to the square-integrable case. The square-integrable case is then reduced by constructing a complete set 
of exact eigenfunctions for the operators. This reduction eliminates the causality constraint and thereby obviates
Wiener-Hopf factorization. The solution of the resulting unconstrained problem is characterized by exploiting
results from the theory of Toeplitz operators on the unit circle \cite{Bottcher2006-nn}.
In particular, we prove bounds on the accuracy and performance of the exact Wiener filter in terms of the
exponent $\alpha$ that describes the power-law of the noise spectra.
Finally, we express our result as a readily-implementable algorithm, which can be applied on 
data without Wiener-Hopf factorization, and demonstrate it in a proof-of-principle
example.
Our method does not rely on any special structure for the 
noises \cite{Thiel49,Nob58,Shin70,Yao71,Reed73} other than that they are weak-stationary,
and is therefore a precise, general, and practical solution of the forecasting problem,
even for scale-free noise.

\newsec{Variational formulation of Wiener filter.}
The Wiener filter \cite{Wien49} is a function $H(t)$
that minimizes the variance of the error, $\eps(t) \eqdef x(t) - (H\ast y)(t)$,
in the estimation of the process $x$ from observation $y$.  
The causal (respectively, acausal) problem demands that the filter $H$ be causal (acausal).
We confine attention to the case where the processes $x,y$ are jointly weak-stationary and are thus 
fully characterized by their power spectral densities $S_{xx}(\omega),S_{yy}(\omega)$ and 
covariance spectral density $S_{xy}(\omega)$. In this case
the variance of the error
\begin{equation}\label{vare}
    V_{\eps}[h] = \int\limits_{-\infty}^{+\infty} S_{\eps \eps}(\omega) \frac{\d{\omega}}{2 \pi},
\end{equation}
has a functional dependence on the Fourier transform $h(\omega)$ of the filter through the power spectrum 
of the error
\begin{equation}\label{eq:See}
\begin{split}
    S_{\eps \eps}(\omega) = S_{xx}(\omega) &+ \abs{h(\omega)}^2 S_{yy}(\omega) \\ 
        &- 2 \Re\{h^\cc(\omega) S_{xy}(\omega)\}.
\end{split}
\end{equation}
Causality of the filter amounts to $H(t)=\Theta(t) H(t)$, where 
the step function $\Theta$ suppresses $H(t<0)$. In the frequency domain,
the causality condition is $h = \ii \mathcal{H}[h]$, where
$\mathcal{H}[f](\omega) \eqdef \int f(\omega')[\pi(\omega-\omega')]^{-1}\dd \omega'$ is the Hilbert transform.

The causal Wiener filter is therefore the solution to the constrained optimization problem
\begin{subequations}\label{Opt}
\begin{align}
    h_* &= \arginf_h V_{\eps}[h]\label{hCausal} \\
    \t{s.t.}\qquad h &= \ii \mathcal{H}[h] \label{hOptCausal}.
\end{align}
\end{subequations}
We look for filters $h$ in the Hilbert space $L^2(\mathbb{R})$. 
Then the constraint, which defines the feasible filters, assumes the form $h \in \ker \mathcal{G}_-$,
where $\mathcal{G}_\pm \eqdef (\mathcal{I}\pm \ii \mathcal{H})/2$ are bounded orthogonal projectors 
such that $\mathcal{G}_+ + \mathcal{G}_- = \mathcal{I}$.

The minimizer $h_*$ is determined by the Gateaux derivative of $V_\eps[h_*]$ in the direction $\delta h$, 
$\delta V_\eps [h_*, \delta h] \eqdef \lim_{\lambda \rightarrow 0} (V_\eps[h_* + \lambda\, \delta h]
-V_\eps[h_*])/\lambda = 0$, as long as the minimizer and the direction are feasible.
Since the map $h\mapsto V_\eps[h]$ is convex, this is necessary and sufficient, and determines the global
minimum (\cite{EkeTem76,Barbu2012-pn} and Supplemental Information \cref{sec:Minimization}). Explicitly,
\begin{equation}\label{eq:deltaVnaive}
	\delta V_\eps[h_*, \delta h] = 2 \Re\, \langle S_{yy} h_* - S_{xy}, \delta h \rangle = 0,
\end{equation}
for $h_*, \delta h \in \ker \mathcal{G}_-$. 
This equation is of the form $\Re \langle g, \delta h\rangle =0$ for $g = S_{yy}h_* - S_{xy}$.
We now use two key properties of $\mathcal{G}_\pm$ to simplify it. 
First, owing to the linearity of $\mathcal{G}_-$, 
if $\delta h \in \ker \mathcal{G}_-$, we also have $\ii \delta h \in \ker \mathcal{G}_-$; so  
if $\Re \langle g, \delta h\rangle = 0$, then 
$\Re \langle g, \ii \delta h\rangle =0$, which implies that $\Im \langle g, \delta h\rangle = 0$,
and therefore $\langle g, \delta h\rangle=~0$.
Second, we have that $\mathcal{G}_+ + \mathcal{G}_- = \mathcal{I}$ and $\mathcal G_\pm$ are orthogonal projectors. 
This implies that $L^2(\mathbb R) = \ker \mathcal G_+ \oplus \ker \mathcal G_-$ and vectors in $\ker \mathcal G_+$ and $\ker \mathcal G_-$ are mutually orthogonal. 
So,  $\langle g,\delta h \rangle = 0$ for all $\delta h \in 
\ker \mathcal{G}_-$ implies that $g \in \ker \mathcal G_+$, leading to $\mathcal{G}_+[g]=0$.
Thus, the necessary and sufficient condition for $h_*$ to be the globally optimal causal
Wiener filter is that
\begin{subequations}\label{Geqs}
\begin{align}
    \mathcal{G}_+[h_* S_{yy}] & = \mathcal{G}_+[S_{xy}] \label{Gplus}, \\
    \mathcal{G}_-[h_*] &= 0 \label{Gminus}.
\end{align}    
\end{subequations}
(In the acausal case, \cref{Gminus} is absent, and the analogue of \cref{Gplus} is $hS_{yy}=S_{xy}$.)
The validity of this conclusion rests on the fact that the Hilbert transform is bounded in 
$L^2(\mathbb{R})$ \cite{Stein1970}; 
in particular, $h$, $hS_{yy}$, and $S_{xy}$ must be in $L^2(\mathbb{R})$, leading to the following observation:

\begin{theorem}\label{Theorem1}
    \Cref{Gplus,Gminus} determine a causal Wiener filter $h$ from the data $(S_{xy},S_{yy})$ if
    the data and the filter belong to $L^2(\mathbb R)$.
\end{theorem}

\newsec{Transformation of scale-free data.}
In this work, we are concerned with spectral data characterized by ``scale-free'' 
asymptotics of the form
$S_{xy}(\omega) = \Theta(1/|\omega|^{\alpha_x})$ and $S_{yy}(\omega) \sim A_y/|\omega|^{\alpha_y}$, $A_y >0$, as $|\omega| \rightarrow \infty$, and $S_{xy}(\omega) = \Theta(1/|\omega|^{\beta_x})$ and $S_{yy}(\omega) = B_y/|\omega|^{\beta_y}$, $B_y>0$, as $|\omega| \rightarrow 0$. 
That is, in general, the data does not belong to $L^2(\mathbb R)$, and so direct application of 
\cref{Gplus,Gminus} to $(S_{xy},S_{yy})$ is invalid.
Assuming that the data $(S_{xy},S_{yy})$ is continuous and non-vanishing away from $\omega = 0, \pm \infty$, 
it is evident that the only obstruction for the data to be square-integrable is its asymptotics in 
the vicinity of $\omega = 0, \pm \infty$. 

To overcome this problem, we for a moment return to the 
variational formulation in \cref{Opt} which seeks a minimizer $h(\omega)$, of $V_\eps[h]$, 
that is analytic in the upper half $\omega$ plane. Considering $V_\eps$ as a function of the filter $h$ and
the data $(S_{xy},S_{yy})$, it is obvious that
\begin{equation}
    V_\eps[h,S_{xy},S_{yy}] = V_\eps[h f^{-1}, S_{xy} f^\cc, S_{yy} \abs{f}^2],
\end{equation}
for any function $f(\omega)$. Thus, the problem of determining the Wiener filter $h$ 
for the data $(S_{xy},S_{yy})$ can be mapped to the Wiener filtering problem for the data
$(S_{xy}',S_{yy}') = (S_{xy}f^\cc, S_{yy}\abs{f}^2)$ for which the filter is $h'=h f^{-1}$, 
for any function $f(\omega)$ devoid of zeros in the upper half plane.
The advantage is that proper choice of $f(\omega)$ can change 
the asymptotics of the transformed data and filter 
in a way that allows to satisfy the conditions of \cref{Theorem1} even if the original data and filter do not. 
For example, $f_{\alpha,\beta}(\omega) = e^{\ii \alpha/2} \omega^\beta (\ii + \omega)^{\alpha-\beta}$ is analytic
and devoid of zeros in the upper half plane, with power-law asymptotics $\abs{\omega}^\beta$ at infinity
and $\abs{\omega}^\alpha$ at zero. Transforming by such a function maps the Wiener filtering problem for
scale-free data into one that is amenable to the formulation in \cref{Geqs}, as
made precise by the following.

\begin{theorem}\label{TheoremBound}
	Suppose the data $(S_{xy}, S_{yy})$ satisfy the following criteria. $S_{xy}, S_{yy}$ are continuous 
    in all closed intervals that exclude 
    $\omega = 0$ and $\omega = \infty$, with asymptotics:
    \begin{equation}
    \begin{aligned}
        &S_{xy}(\omega \rightarrow \infty) = O\left(|\omega|^{-\alpha_x}\right),\,
        &S_{yy}(\omega \rightarrow \infty) \sim \frac{A_y}{|\omega|^{\alpha_y}}, \\
        &S_{xy}(\omega \rightarrow 0) = O(|\omega|^{-\beta_x}),\,
        &S_{yy}(\omega \rightarrow 0) \sim \frac{B_y}{|\omega|^{\beta_y}},
    \end{aligned}
    \end{equation}
    where $A_y,B_y >0$, $2 \alpha_x - \alpha_y > 1$ and $2 \beta_x - \beta_y < 1$. Additionally, $S_{yy}$ is strictly positive away from $\omega = 0, \pm \infty$.
    Then, the modified data $(S_{xy}',S_{yy}')$ corresponding to $|f(\omega \rightarrow \infty)|
    \sim |\omega|^{\alpha_y/2}$ and $|f(\omega\rightarrow 0)| \sim |\omega|^{\beta_y/2}$
, has the following asymptotics
    \begin{equation}\label{eq:Thrm2Asymp}
    \begin{aligned}
        &S_{xy}^\prime(\omega \rightarrow \infty) = o(\abs{\omega}^{-1/2})\,
        &S_{yy}^\prime(\omega \rightarrow \infty) \sim A_y\\
        &S_{xy}^\prime(\omega \rightarrow 0) = o(\abs{\omega}^{-1/2})\,
        &S_{yy}^\prime(\omega \rightarrow 0) \sim B_y;
    \end{aligned}    
    \end{equation} 
    and, the optimal causal filter $h'$ satisfies \cref{Gplus,Gminus} with the modified data 
    and is in $L^2(\mathbb{R})$. Additionally, $S_{yy}^\prime$ is a positive function bounded from above and below by positive numbers.
\end{theorem}

A detailed proof is available in Supplemental Information \cref{sec:ProofTHeorem2}. 
A crucial element of the proof is the observation that owing to the Cauchy-Schwarz inequality, 
$S_{xx} S_{yy} \geq |S_{xy}|^2$; and in view of \cref{vare}, 
$S_{xx}$ itself needs to be integrable, so that the asymptotic growth indices
of $S_{xy}$ and $S_{yy}$ have to satisfy $2 \alpha_x - \alpha_y > 1$ and $2 \beta_x - \beta_y < 1$. 
The proof also does not clarify whether the solution for the transformed problem, via \cref{Gplus,Gminus}, 
is square-integrable; this can be separately verified using techniques from the theory of Toeplitz 
operators (see Supplementary Information \cref{sec:HpSpace}).


\newsec{Eigenfunctions of $\mathcal{G}_\pm$.}
The procedure laid out above, leading to \Cref{TheoremBound}, shows how a given problem can be transformed
into a form where the data $(S_{yy},S_{xy})$ satisfies the conditions of \cref{Theorem1}. In the sequel, we
assume this has been done. 
Then, the causal Wiener filter is the solution of the pair of operator equations \cref{Geqs}.
Since $\mathcal{G}_{\pm}$ are orthogonal projectors, their eigenfunctions together span some subspace
of $L^2(\mathbb{R})$. Once these eigenfunctions are determined, the causal Wiener filter problem
can be explicitly solved in their span. The precise characterization of the eigenfunctions will therefore solve 
the causal Wiener filter problem of \cref{Geqs}.

We begin by observing that the eigenfunctions of $\mathcal{G}_\pm$ are
eigenfunctions of the Hilbert transform with eigenvalues $\pm \ii$. 
The following theorem characterizes one such set.

\begin{theorem}\label{thm:eigHilbert}
The sequence of functions
\begin{equation}\label{eq:eigHilbert}
	\phi_k (\omega) = \frac{\pi^{-1/2}}{1 - \ii \omega} \left( \frac{1 + \ii \omega}{1 - \ii \omega} \right)^k , \qquad (k \in \mathbb{Z})
\end{equation}
is a complete orthonormal set in $L^2(\mathbb R)$
and are eigenfunctions of the Hilbert transform:
\begin{equation}\label{eq:Heig}
	\mathcal{H}[\phi_k] = - \ii\, \t{sgn}(k)\cdot \phi_k,
\end{equation}
where $\t{sgn}(k \geq 0) \eqdef +1$, $\t{sgn}(k<0) \eqdef - 1$.
Any $~g~\in~L^2(\mathbb R)$ can be expanded in $\{\phi_k\}$:
\begin{equation}\label{eq:fsum}
	g(\omega) =\sum_{k = -\infty}^{+\infty} g_k \phi_k(\omega), 
\end{equation}
with convergence understood in the sense of pointwise almost everywhere
with $g_k = \langle \phi_k, g\rangle_{L^2(\mathbb{R})}$.
\end{theorem}

The proof of \cref{eq:Heig} is an explicit computation by contour integration. 
The form of $\phi_k$, and indeed its completeness, follows from its relation to Fourier series:
the variable substitution $\omega = \tg (u/2)$ maps the real line $\mathbb{R}$ into the interval 
$[-\pi, +\pi]$ with the ends identified, i.e. the unit circle $\mathbb T$.
The set of functions $\{e^{\ii k u}\mid k\in \mathbb{Z}\}$ is a complete orthonormal basis in $L^2(\mathbb{T})$;
$\phi_k$ is its inverse image under the variable substitution. 
The variable substitution induces the map $g(\omega) \mapsto \tilde{g}(u) = \sqrt{\pi} e^{\ii u/2} 
g(\tan (u/2)) \sec(u/2)$. The expansion coefficient $g_k$ of \cref{eq:fsum} is then the $k^\t{th}$ 
Fourier coefficient of $\tilde{g}(u)$. 
Thus, the map $g(\omega) \mapsto \tilde{g}(u)$ is an isometric isomorphic map 
$L^2(\mathbb R) \rightarrow L^2(\mathbb T)$. As a result, the basis $\{\phi_k(\omega)\}$ inherits the completeness
of the Fourier basis $\{e^{\ii k u}\}$, and the convergence of \cref{eq:fsum} almost everywhere follows 
from the convergence of the Fourier series 
almost everywhere in $L^2(\mathbb T)$ \cite{carleson1966convergence,Hunt1968,Jorsboe1982-ef}.
See \cref{app:eigHilbert} for details.

The immediate corollary of \cref{thm:eigHilbert} is that
\begin{equation}\label{eq:eigG}
	\mathcal{G}_+[\phi_k] = \begin{cases} \phi_k; & k \geq 0 \\ 0; & k < 0  \end{cases}, 
	\qquad
	\mathcal{G}_-[\phi_k] = \begin{cases} 0; & k \geq 0 \\ \phi_k; & k < 0  \end{cases}. 
\end{equation}
That is, in $L^2(\mathbb R)$, the causality constraint [\cref{Gminus}] can be solved exactly:
\begin{equation}\label{eq:hexpansion}
	h(\omega) = \sum_{k = 0}^{\infty} h_k \phi_k(\omega)
\end{equation}
with equality being true almost everywhere for some sequence of coefficients $\{h_k\}$ to be determined by \cref{Gplus}. Given that $S_{xy}$ and $S_{yy}h$ also belong to $L^2(\mathbb R)$,
\cref{Gplus} is equivalent to
\begin{equation}\label{eq:GhProj}
  \langle \phi_k , \, \mathcal{G}_+[h S_{yy}-S_{xy}] \rangle_{L^2(\mathbb{R})} = 0
\end{equation}
for all integer $k$. 
The isomorphism between $L^2(\mathbb{R})$ and $L^2(\mathbb{T})$ and
the convolution theorem of Fourier series in $L^2(\mathbb{T})$ can be employed to reduce the above
equation to the set of linear equations
\begin{equation}\label{eq:linearEq}
    \sum_{n = 0}^{\infty} t_{k-n} h_n = S_{xy, k}, 
\end{equation}
where $h_k = \langle \phi_k, h\rangle, S_{xy,k} = \langle \phi_k, S_{xy}\rangle$, and 
$2\pi t_k = \int_{- \pi}^{\pi} \d{u} e^{i k u} S_{yy}\left(\tan u/2\right)$.
The system of linear equations (\ref{eq:linearEq}) can be then written as $\vec{T h} = \vec{s}$, where 
$\vec{T}$ is a symmetric Toeplitz matrix with entries $t_k$.
The formal solution of this infinite Toeplitz system, which can be shown to be unique 
almost everywhere (see Supplemental Information \cref{sec:Tpositive}), determines the causal Wiener filter. 
Clearly, this can be done without explicit Wiener-Hopf factorization, since the causality constraint has been
explicitly solved.

\newsec{Quality of approximate solution.} A practical solution of \cref{eq:linearEq} calls for a 
robust and accurate finite-dimensional approximation of it. 
The simplest is truncating $\vec{T}$ to the $n\times n$ matrix $\vec{T}^{(n)}$. 

First, we establish that the solution $\vec h^{(n)}$ of the truncated system is unique for every $n$ and 
that its numerical evaluation is robust.
The uniqueness of $\vec h^{(n)}$ follows from the positivity and therefore invertibility of $\vec T^{(n)}$ (see Supplemental Information \cref{sec:Tpositive}). 
The numerical robustness of the evaluation of $\vec h^{(n)}$ is controlled by the 
condition number $\kappa(\vec T^{(n)})$ --- the ratio of the largest to the smallest eigenvalues of $\vec{T}^{(n)}$. 
The best upper bound on the condition number of a Toeplitz matrix is 
\cite{BOTTCHER1998285,BOGOYA20151308,BOGOYA2016606} $\kappa(\vec T^{(n)}) \leq \sup[S_{yy}(\omega)]/\inf[S_{yy}(\omega)]$, which is valid only for $\inf S_{yy} > 0$
\footnote{
The bound from below can be established under general assumptions and takes the form
\begin{equation}\label{Kappa1}
	1+ \frac{v(\vec T_n)}{t_0} \leq \kappa(\vec{T}^{(n)}),
\end{equation}
where $v(\vec T^{(n)})^2 = 2 \sum_{k=1}^{n-1} t_k^2(n-k)/n$.
This lower bound is a direct consequence of bounds on eigenvalues based on the trace of the 
matrix \cite[Corollary 2.3]{WOLKOWICZ1980471}.
The bound on $\kappa(\vec T^{(n)})$ from above is 
\begin{equation}
	\kappa(\vec T^{(n)}) \leq \sup [S_{yy}(\omega)]/\inf [S_{yy}(\omega)] < \infty,
\end{equation}
where $\sup [S_{yy}(\omega)]$ and $\inf [S_{yy}(\omega)]$ are the upper and lower bounds on the maximal and minimal eigenvalues of $\vec T^{(n)}$ correspondingly (see Supplementary Information \cref{sec:Tpositive}).
}. 
The positivity of $\inf S_{yy}(\omega)$ is guaranteed by the choice of $f(\omega)$ that was made 
in \cref{TheoremBound}.

Since the approximate solution $\vec h^{(n)}$ is unique and robust, we proceed to show 
that it is accurate in the sense that $\norm*{\vec h^{(n)} - \vec h} \rightarrow 0$ 
as $n \rightarrow \infty$.
Let $\vec h^{(n)} = \{ h_k^{(n)} \} $, such that
the ``truncated" Wiener filter is $h^{(n)}(\omega) = \sum_{k = 0}^n h_k^{(n)} \phi_k(\omega)$. 
Establishing $\norm*{\vec h^{(n)} - \vec h} \rightarrow 0$ is equivalent to establishing 
$\norm*{h^{(n)}(\omega) \rightarrow h(\omega)}_{L^2(\mathbb{R})} \rightarrow 0$. 
Two types of terms contribute to the difference between $h^{(n)}(\omega)-h(\omega)$. 
The first type arises from a truncation of the sum that defines $h^{(n)}(\omega)$ over $k$ at finite $n$. 
The second contribution comes from the difference between the approximate and exact values of the 
coefficients $h_k$ and $h_k^{(n)}$ for $k \leq n$. 
Both contributions can be bounded using Parseval's theorem~\cite{Zygmund_2003} and shown to have a vanishing 
norm as $n \rightarrow \infty$. 
Thus the following theorem is established (see Supplementary Information \cref{sec:ProofTheorem4} for 
details).

\begin{theorem}\label{SolutionConvergence}
    Let $S_{xy}(\omega)$ and $S_{yy}(\omega)$ be such that $S_{xy}$ is an $L^2(\mathbb R)$ function, and $S_{yy}$ is a continuous, positive function that is bounded by positive numbers from above and below. Then, the causal filter $h(\omega)$ corresponding to data $(S_{xy}, S_{yy})$ is a member of $L^2(\mathbb R)$ and is unique.  
    The coefficients $h_k$ ($p\geq 0$) that determine the causal Wiener filter $h(\omega) = \sum_k h_k 
    \phi_k(\omega)$ are unique. The truncated filter $h^{(n)}(\omega)$, which solves the $n \times n$ truncation 
    of \cref{eq:linearEq}, converges to $h(\omega)$ in $L^2(\mathbb R)$ norm. Its expansion coefficients 
    $\{h_k^{(n)}\}$ converge to $h_k$ for every fixed $k$. 
\end{theorem}

\cref{TheoremBound,SolutionConvergence} together constitute a powerful tool that enables the 
robust and accurate computation of the causal Wiener filter without the need for Wiener-Hopf factorization. 
Better convergence of the approximate solution can be established if the data $(S_{xy},S_{yy})$ satisfy stronger 
conditions. An example is the following theorem, which realizes the intuition that the filter inherits similar
smoothness properties as the data. 

\begin{theorem}\label{trm:WienerAndSmoothWiener}
Let the data $(S_{xy},S_{yy})$ satisfy conditions of \cref{SolutionConvergence}. 
In addition, suppose that $\sum_{n = 0}^{\infty} |t_n|$ and $\sum_{n = 0}^{\infty} |S_{xy, n}|$ are both finite. Then, (i) the expansion $h(\omega) = \sum_{k\geq 0}h_k \phi_k (\omega)$ converges absolutely and uniformly; 
(ii) the truncated filter $h^{(n)}(\omega)$ converges to $h(\omega)$ everywhere uniformly; and, (iii) the 
coefficients $h_k^{(n)}$ converge to $h_k$ uniformly in $k$ as $n \rightarrow \infty$.

Suppose, in addition, there exists $0<\omega_0 < \infty$, such that data are of a H\"{o}lder class $C^{m, \alpha}$ for integer $m \geq 1$ and  $0 < \alpha <1$ on $[-\omega_0, \omega_0]$, and $S_{yy}(1/\omega)$ and $(1+ \ii/\omega)S_{xy}(1/\omega)$ are of a H\"older class $C^{m, \alpha}$ on $[-\omega_0^{-1}, \omega_0^{-1}]$.
Then, (i) $h(\omega)$ is also of a H\"{o}lder class $C^{m, \alpha}$ on $[- \omega_0, \omega_0]$ and $h(1/\omega)$ is of a H\"older class $C^{m, \alpha}$ on $[-\omega_0^{-1}, \omega_0^{-1}]$; 
(ii) $h_k = O(1/k^{m+\alpha})$;
(iii) the convergence of the truncated filter is controlled as 
$|h^{(n)}(\omega) - h(\omega)| < A n^{1-m-\alpha}$ for some $A >0$; and,
(iv) the convergence of the expansion coefficients is uniform in $k$ and controlled by 
$|h^{(n)}_k - h_k| < B n^{1-m-\alpha}$ for some $B>0$. 
\end{theorem}

The proof of this theorem, described in Supplemental Information \cref{sec:WienerAndSmoothWiener}, is not simple. 
It is however interesting that satisfying the conditions of \cref{trm:WienerAndSmoothWiener} requires the data to obey $2 \alpha_x - \alpha_y > 2$ and $2 \beta_x - \beta_y < 0$, which are stronger than the conditions
necessary for \cref{TheoremBound}.

\newsec{Finite-size effect of experimental data.}
In any experiment, the data $(S_{xy}(\omega), S_{yy}(\omega))$ is only ever known in a
finite frequency range, i.e. for $\omega \in [\omega_m,\omega_M]$. So far however, we assumed 
access to data at all frequencies. In order to reconcile with experimental reality, we now construct
a procedure to estimate and minimize the finite-size error arising in $\vec h^{(n)}$ from truncation 
in frequency space.
As before, we assume that the data $(S_{xy}, S_{yy})$ has undergone the transformation leading 
up to \cref{TheoremBound}.

The finite size effect produces a systematic error in $\vec h^{(n)}$ via errors in $\vec T^{(n)}$ 
and $\vec s$. We will now estimate both contributions, and show that it can be minimized by exploiting 
the frequency-scale of the data, which can be freely done for scale-free data. 
An arbitrary frequency scale $\omega_0$ implicitly exists in the definition of the variable substitution
$\omega = \omega_0 \tg(u/2)$, that maps the real frequency line to the unit circle. In the preceding
discussion, this scale $\omega_0$ was implicit, in the following, we make it explicit.

First we focus on the error due to those in $\vec{T}^{(n)}$. With finite-size data, its elements are given by
\begin{equation}\label{eq:tkCutoff}
	t_k = 2\int_{u_m}^{u_M} \frac{\d u}{2 \pi} \cos({k u}) \, S_{yy}\left(\omega_0\tan \frac{u}{2}\right),
\end{equation}
where $u_{m,M} = 2\tan^{-1}(\omega_{m,M}/\omega_0)$.
Assuming $\omega_m \ll \omega_0 \ll \omega_M$, we have that, $u_m \approx 2 \omega_m/\omega_0$, and  
$u_M \approx (\pi - \omega_0/\omega_M)$. 
The difference between $t_k$ evaluated in \cref{eq:tkCutoff} and its exact value is denoted as $\delta t_k$ and has a simple upper bound 
\begin{equation}\label{eq:estimate}
	|\delta t_n| < \frac{1}{\pi} \sup(S_{yy})(u_m + (\pi - u_M)).
\end{equation}
The finiteness of this bound is ensured by the finiteness of $\sup(S_{yy})$ guaranteed by \cref{TheoremBound}
for the transformed data.
Since cut-offs $u_m$ and $u_M$ are not independent, varying $\omega_0$ allows to minimize the upper bound in 
\cref{eq:estimate}: the optimal choice is $\omega_0 = \sqrt{\omega_m \omega_M}$, resulting in 
\begin{equation}\label{eq:tnerror}
|\delta t_n| \leq \frac{4}{\pi} \sup(S_{yy}) \sqrt{\frac{\omega_m}{\omega_M}}.
\end{equation}
Denoting the error in $\vec h^{(n)}$ that is inflicted by $\delta t_n$ as $\delta \vec h_y^{(n)}$, the relative error can be estimated as
\begin{equation}\label{eq:serror}
	\frac{||\delta \vec h_y^{(n)}||}{||\vec h^{(n)}||} \leq \frac{n |\delta t_n|}{\inf S_{yy}(\omega)} 
    \leq \frac{4 n \kappa(\vec{T}^{(n)})}{\pi} \sqrt{\frac{\omega_m}{\omega_M}}.
\end{equation}
The relative error is small when the truncation size of the matrix $\vec{T}$ 
is $n \ll n_{\max} =  \kappa^{-1}\sqrt{\omega_M/\omega_m}$. 
This leads to a practical prescription: the size $n$ of the truncated matrix $\vec T^{(n)}$ should not 
exceed $n_{\mathrm{max}}$.
This result is consistent with a basic intuition of Fourier analysis: data contains no meaningful information in the frequency range below the measurement duration and above the sampling rate.

\begin{figure*}[t!]
    \centering
    \includegraphics[width=1\linewidth]{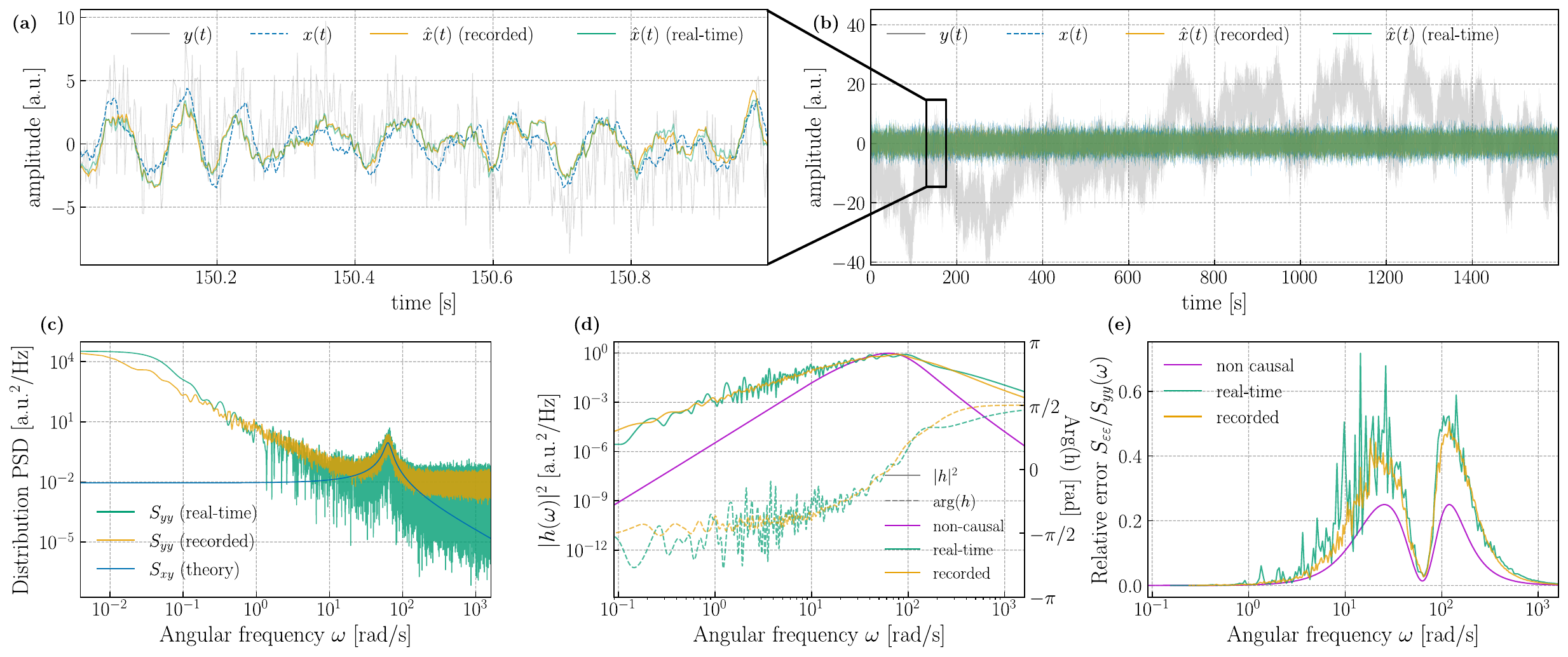}
    \caption{\label{fig}\textbf{Causal Wiener filtering of scale-free data.}  
    (a,b) Grey shows a simulated trajectory of the noisy measurement $y(t) = x(t) + n(t)$ for $1600$ s, sampled at 
    $512$. These simulations assume an underlying signal with spectrum
    assumed to be $S_{xx}=A \gamma^2/ ((|\omega|-\omega_c)^2+\gamma ^2)$ with 
    parameters $\gamma = 2 \pi$ rad/s, $A=0.9$, and $\omega_c = 10 \cdot 2\pi$ rad/s -- a peak centered at $2\pi \cdot 10$ rad/s with 
    half-width $2 \pi$ rad/s. 
    The noise was chosen such that $S_{nn}=5/\omega^{1.8}+0.01$. 
    (a) Zoomed-in time series of the measured signal $y(t)$, the true signal $x(t)$, and estimate $\hat{x}(t)$ from the recorded and real-time causal filters.
    (c) Estimates of the power spectral density $S_{yy}$ obtained from the simulated time series in (a) using
    Welch's method, with a Slepian window, 80–90\% segment overlap, and segment length chosen so that approximately eight Fourier bins spanned the target linewidth
    $S_{xy}$ is presumed known from a model.
    (d) Causal Wiener filter in the ``recorded'' and ``real-time'' scenarios (orange and green respectively),
    compared against the non-causal Wiener filter $S_{xy}/S_{yy}$. Solid lines show squared magnitude and dashed
    line the phase of the filter.
    The causal filters are computed following the algorithm laid out in the main text: data is transformed as prescribed by \cref{TheoremBound} with $f(\omega) = e^{\ii a /2} \omega^\beta (\ii \omega_0 + \omega)^{\alpha - \beta}$ with $\alpha=0$, $\beta=0.9$, and $\omega_0 = 5$ rad/s. 
    The scale frequency in \cref{eq:tkCutoff} was also chosen to be $5$ rad/s.
    (e) Relative error PSDs $S_{\eps\eps}/S_{yy}$ for the two causal filter scenarios and the non-causal filter. 
    $S_{\eps \eps}$ was estimated by Welch’s method to the residual $x(t)-\hat{x}(t)$, averaging 
    across approximately 250 logarithmically spaced bins in $\omega$, with outliers more than 10 
    standard deviations removed.
}
\end{figure*}

To estimate the error in $\vec h^{(n)}$ due to finite-size effects of $S_{xy}$, we first estimate the corresponding error $\delta \vec s$ arising in $\vec s$.
Such an error arises from the change in integral bounds in $s_n = \langle \phi_n, S_{xy} \rangle$, which can be bounded as follows.
\cref{TheoremBound} guarantees that 
there exist $0 \leq \bar B_x, \bar A_x < \infty$, such that $S_{xy}(\omega) \leq \bar B_x/|\omega|^{\beta_x - \beta_y/2}$ and  $S_{xy}(\omega) \leq \bar A_x/|\omega|^{\alpha_x - \alpha_y/2}$ in the vicinity of $\omega = 0$ and $\omega = \infty$ correspondingly. Then $\delta s_n$ can be bounded as
\begin{multline}
	|\delta s_n| < \frac{1}{\pi} \left( \frac{\bar B_x}{(\beta_x - \beta_y/2-1)u_m^{\beta_x - \beta_y/2-1}} + \right.
	\\
	 \left. + \frac{\bar A_x (\pi - u_M)^{\alpha_x - \alpha_y/2}}{(\alpha_x - \alpha_y/2)}\right).
\end{multline}
The resulting error in $\vec{h}^{(n)}$, denoted $\delta \vec h_x^{(n)}$, 
satisfies
\begin{equation}\label{eq:relhxerror}
	\frac{||\delta \vec h_x^{(n)}||}{||\vec h^{(n)}||} 
    \leq \kappa(\vec{T}^{(n)}) \frac{||\delta \vec s||}{||\vec s||} 
    \leq \sqrt{n} \kappa(\vec{T}^{(n)}) \frac{|\delta s_n|}{||\vec s||}.
\end{equation}

Comparing $\delta \vec h_y^{(n)}$ and $\delta \vec h_x^{(n)}$ against each other and against $\vec h^{(n)}$ is important for optimal selection of $\omega_0$. 
When $2 \beta_x - \beta_y< 0$ and $2 \alpha_x - \alpha_y >2$, the error $\delta h_x^{(n)}$ vanishes faster than $\delta h_y^{(n)}$ when $\omega _m \rightarrow 0$ and $\omega_M \rightarrow \infty$.
In this case, $\omega_0 = \sqrt{\omega_m \omega_M}$ remains an asymptotically optimal choice of the frequency scale. 
In the opposite case, the behaviour of both errors must be carefully examined case by case to minimize both of the errors.

\newsec{Filtering data of $1/\omega^\alpha$ form.}
Based on \cref{Theorem1,TheoremBound,thm:eigHilbert,SolutionConvergence,trm:WienerAndSmoothWiener}, we 
can now propose a powerful algorithm to compute the causal Wiener filter, with performance guarantees, 
for a large class of data $(S_{xy},S_{yy})$. The algorithm is as follows: 
\begin{enumerate}
\item Transform the original scale-free data as in \cref{TheoremBound} to obtain the modified 
    data $(S_{xy}, S_{yy})$.
\item Compute the matrix $\vec T^{(n)}$ and $\vec s$ using \cref{thm:eigHilbert}, 
    evaluate $\vec h^{(n)}$ for some $n$ and data $(S_{xy}, S_{yy})$.
\item Use \cref{SolutionConvergence,trm:WienerAndSmoothWiener} to obtain bounds on the difference between 
    $h^{(n)}(\omega)$ and the exact causal Wiener filter corresponding to $(S_{xy}, S_{yy})$.
\item Estimate $n_{\max}$ and errors from \cref{eq:serror,eq:relhxerror} to quantify the effects of finite data length.
\item Reconstruct the approximate causal Wiener filter for original data from $h^{(n)}(\omega)$.
\end{enumerate}
This procedure extends the reach of forecasting to a general class of problems relevant in science
and engineering, and especially the ubiquitous class of scale-free data that was not systematically 
amenable to prior techniques.

As an illustration of the above theory, we consider a typical example of forecasting a signal $x(t)$ 
from its noisy measurement $y(t)=x(t)+n(t)$, where $n(t)$ is noise with non-rational spectrum. 
The grey curve in \cref{fig}(a,b) shows a simulated realization of $y(t)$ corresponding to an assumed
signal and noise spectra.
The signal and noise were chosen to reflect forms representative of turbulence \cite{BruCarb13},
neural processes \cite{neuro_noise}, or mechanical oscillators in precision physics 
experiments \cite{GroAsp15,FedKip18,CripCorb19}: the noise
has a scale-free and a white component, and the signal
has a Lorentzian spectrum (see \cref{fig} caption for details). 

Simulating the downstream data processing procedure, we estimate the power spectrum $S_{yy}$ from the measured
$y(t)$ using Welch's method \cite{welch}. The cross-correlation spectrum $S_{xy}$ is assumed known from a model;
we chose, without loss of generality, that the signal and noise were orthogonal so that $S_{xy} = S_{xx}$.
This data $(S_{xy},S_{yy})$ is used to compute the causal Wiener filter following the approach described above.
As prescribed by \cref{TheoremBound}, the data is transformed by the function $f(\omega) = e^{\ii a /2} 
\omega^\beta (i \omega_0 + \omega)^{\alpha - \beta}$ (see \cref{fig} caption for details), 
resulting in a well-conditioned Toeplitz system. 
The system of linear equations was solved and the filter was reconstructed from a basis of 100 
eigenfunctions (more eigenfunctions produced negligible improvement).

Two scenarios of operations were considered. The ``recorded'' scenario presumes access to a long-run record of 
the data $y(t)$, from which the causal Wiener filter is computed, which is then applied to a different 
realization of $y(t)$. In the second, "real-time" scenario, the filter was constructed and applied sequentially to the same data stream by dynamically evaluating PSDs and the corresponding filter. The green (orange) curve
in \cref{fig}(a,b) shows the estimated signal $\hat{x}(t)$ in the real-time (recorded) scenario. These estimates
show qualitative agreement with the underlying signal $x(t)$ (blue dashed). \Cref{fig}(c) shows data $(S_{xy},
S_{yy})$ used to compute the filter, while \cref{fig}(d) shows the magnitude response of the filter. 
To ascertain goodness of forecast, \cref{fig}(e) depicts the spectrum of the error normalized to the spectrum
of measurement. 

Variance in the estimated error spectra, visible in the green trace in \cref{fig}(e), disappears when the filter
is computed from long-run data (orange). 
However, the filter $h(\omega)$ converged rapidly and robustly even when estimated from relatively short
lengths of measurement $y(t)$, as shown in \cref{fig}(d).
By computing the Wiener filter on the analytic form of the PSDs in Supplemental Information \cref{sec:ExtraData}, we also show that the sampling procedure applied to $S_{yy}$ had no significant effects on the filter performance, besides introducing sampling noise into the filter itself. In Supplemental Information \cref{sec:ExtraData}, we also show that the naive, and computationally-intensive, solution of \cref{Gplus,Gminus} on a lattice coincides with that 
obtained from an expansion into eigenfunctions. Finally, we empirically observed that when the signal-to-noise
ratio was small, transforming the data with the proper scaling function $f(\omega)$ improved accuracy 
and speed of convergence. In the opposite regime, since the impact of noise is minimal, not applying the transformation enhanced convergence. 

\textbf{Conclusion.} The problem of forecasting signals from noisy measurements where the noise is described
by non-rational or even scale-free spectra is an ubiquitous problem in science and engineering. So far, a 
systematic approach to the problem, and characterization of the solution, had remained elusive. We described
a systematic approach that exploits a variational formulation of the causal Wiener filter, and a simplifying
symmetry in that formulation. The symmetry allows a transformation of the problem to a pair of linear
operator equations in $L^2$. One of these equations imposes the causality constraint, which we explicitly solve
by diagonalizing the relevant operator in $L^2$. Thus, the conventional Wiener-Hopf factorization step is
obviated. The problem is thus reduced to the solution of a Toeplitz linear system. Various
practically relevant performance guarantees are derived for the solution of this system, including precise
conditions on the data that ensure the existence and convergence of the filter, and the effect of finite-length
data on the estimated filter.
In sum, we extend the reach of forecasting to the hitherto inaccessible realm of signals contaminated by noise with
scale-free and/or non-rational spectra. Finally, by the duality between estimation 
and control \cite{KalBuc61,YouBonJab76,MitNew00}, the technique outlined here for causal estimation from noise with 
non-rational spectra can be employed to design controllers for stochastic distributed systems.
These should find applications in domains as disparate as neuroscience, 
fluid dynamics, geophysics, quantum physics, and even finance. 

\textbf{Acknowledgements.} 
We thank Manuel Bogoya for a discussion on the theory of Toeplitz operators, and George Verghese
for his general wisdom.
This research is funded in part by an NSF CAREER award (PHY–2441238) and by 
the Gordon and Betty Moore Foundation (grant GBMF13780).

\bibliography{../refs_wiener}


\clearpage
\appendix

\renewcommand{\thetheorem}{\thesection.\arabic{theorem}}
\setcounter{theorem}{0}

\section{Minimization of a convex functional on a complex domain}\label{sec:Minimization}

The arguments leading up to \cref{Geqs} technically involves the minimization of a functional ($V_\eps[h]$) 
whose domain is the set of complex functions ($h$). It is more common \cite{EkeTem76,Barbu2012-pn} to consider
the minimization of a functional whose domain is the set of real functions. Here we show how to bridge this gap.
First, we define a complex-$\cc$ Banach space, which will be useful below.
\begin{definition}
	Complex-$\cc$ Banach space $X$ is a Banach space over the complex field that is equipped 
    with a conjugation operation $\cc$.
	The~$\cc$~operation is (i) a continuous involution on $X$ that obeys $(c_1 x_1 + c_2 x_2)^\cc = c_1^\cc x_1^\cc 
    + c_2^\cc x_2^\cc$ for all $x_{1,2} \in X$ and $c_{1,2} \in \mathbb C$; and
    (ii) $||x^\cc||_X = ||x||_X$ for all $x \in X$.
\end{definition}

Given a complex-$\cc$ Banach space $X$, a real Banach subspace can be constructed as follows. Define the
continuous linear operators $\Re[x] \eqdef (x + x^*)/2$ and $\Im[x] = (x-x^*)/(2\ii)$ for all $x \in X$. 
Then the space $X_r \eqdef \{ x \in X \mid \Im[x] = 0 \}$ is a real linear vector subspace of $X$, which
is complete (since it is the kernel of a continuous linear operator Im \cite[Theorem 5.25]{Hunter2001-zm}),
and has a norm induced by that on $X$; i.e. $X_r$ is a real Banach subspace of $X$. 
A useful property of complex-$\cc$ Banach spaces is that it can be represented as two copies of real Banach spaces.
\begin{lemma}
	Let $X$ be a complex-$*$ Banach space, and let $X_r = \Re[X]$, a real Banach space. 
    Spaces $X$ and $X_r \oplus X_r$ are isomorphic as linear spaces, with the isomorphism given 
    by $x \rightarrow (\Re[x], \Im[x])$. 
    Moreover, the norm $||x||_X$ is equivalent to $||\Re[x]||_{X_r}+||\Im [x]||_{X_r}$ for all $x \in X$. 
\end{lemma}

The isomorphism can be checked explicitly.
The equivalence of norms follows from $||x||_{X} \leq ||\Re [x]||_{X} + ||\Im [x]||_X \leq 2||x||_X$. 

The following definitions and theorems describes the notion of a (Gateaux) derivative of a functional 
on a Banach space, and how that characterizes the minimum of the functional when it is convex.
\begin{definition}
Let $X$ be a complex-$\cc$, and let  $f : X \rightarrow \mathbb R \cup \{+\infty\}$ be a proper convex functional. 
The Gateaux differential of $f$ in the direction of $h \in X$ at the point $x \in X$ is
\begin{equation}
	\delta f[x; h] \eqdef \lim_{\lambda \rightarrow 0} \frac{f[x+ \lambda h] - f[x]}{\lambda}, \quad \lambda \in \mathbb R.
\end{equation}

If $X$ is a real Banach space, the functional $f$ is called Gateaux-differentiable at point $x$, 
if the map $h \mapsto \delta f[x; h]$ is a continuous linear functional on $X$.
The functional $f$ is called Gateaux-differentiable if it is Gateaux-differentiable at all points $x \in X$.

If $X$ is a complex-$\cc$ Banach space, the functional $f$ is called conformally Gateaux-differentiable at 
point $x$, if there exists a unique continuous complex-linear functional $h \mapsto \delta f'[x, h]$, such 
that $\delta f[x; h] = 2\Re \delta f'[x; h]$ for all $h$.
The functional $f$ is called conformally Gateaux-differentiable if it is conformally Gateaux-differentiable at all points $x \in X$.
\end{definition}

The following theorem governs the properties of the minima of Gateaux differentiable functionals.
\begin{theorem}[Proposition 2.40, \cite{Barbu2012-pn}]\label{thrm:GateauxDifReal}
Let $X$ be a real Banach space and let $f : X \rightarrow \mathbb R$ be a proper convex Gateaux-differentiable functional.
Then $f[x] = \inf_{u \in X} f[u]$ for $x \in X$, if and only if $ \delta f[x; h] = 0$ for all $h \in X$.
\end{theorem}

The notion of complex-$\cc$ Banach spaces allows for a trivial generalization of \cref{thrm:GateauxDifReal} in such a way that allows to avoid dealing with the nuissance of $\delta f[x; h]$ not being a complex-linear functional.
\begin{theorem}\label{thrm:GateauxDifComplex}
	Let $X$ be a complex-$\cc$ Banach space and let $f : X \rightarrow \mathbb R$ be a proper convex, conformally Gateaux-differentiable functional.
Then $f[x] = \inf_{u \in X} f[u]$ for $x \in X$, if and only if $\delta f'[x, h] = 0$ for all $h \in X$.
\end{theorem}
\begin{proof}
The idea of the proof consists in mapping the problem from $X$ into $X_r \oplus X_r$, and then applying \cref{thrm:GateauxDifReal}.
Since $X$ is a complex-$\cc$ space, it is isomorphic to $X_r \oplus X_r$.
Under the isomorphism, we define a functional $\bar f: X_r \oplus X_r \rightarrow \mathbb R$, such that $\bar f[x, y] = f[x + \ii y]$. 
Functional $\bar f$ is a proper-convex functional on $X_r \oplus X_r$ and $\bar f[x, y]$ achieves minimum if and only if $f[x+\ii y]$ does so.
Moreover, $\bar f$ is Gateaux-differentiable on $X_r \oplus X_r$ if $f$ is conformally Gateaux-differentiable on $X$.
Indeed, for $h = h_x + \ii h_y$, such that $h_x, h_y \in X_r$, the Gateaux differential of $\bar f$ is
\begin{equation}\label{eq:DiffCon}
	\delta \bar f[x, y; h_x, h_y] = 2 \Re(\delta f'[x+ \ii y, h_x + \ii h_y]).
\end{equation}
which is a continuous real-linear functional of $(h_x, h_y)$ due to continuity and complex-linearity of $\delta f'$.
Therefore, by \cref{thrm:GateauxDifReal}, point $(x, y)$ is an infimum of $\bar f[x, y]$ if and only if $\delta \bar f[x, y, h_x, h_y] = 0$ for all $h_x, h_y \in X_r$.

Now we establish that $\delta f'$ vanishes if and only if $\delta \bar f$ vanishes.
First, $\Re(\delta f'[x+ \ii y, h_x + \ii h_y]) = 0$ for all $h_x, h_y \in X_r$ due to \cref{eq:DiffCon}.
On the other hand, $\Im (\delta f'[x+ \ii y, h_x + \ii h_y]) = \Re (\delta f'[x+ \ii y, -\ii h_x + h_y]) = 0$,
since $\Re (\delta f'[x+ \ii y, -\ii h_x + h_y]) = \delta \bar f[x, y; h_y, - h_x]$, and $\delta \bar f$ vanishes identically.
Therefore, $\delta f'[x+\ii y, h]$ vanishes if and only if $\delta \bar f[x, y, h_x, h_y]$ does.

As a result, $\delta f'[x+\ii y, h]$ vanishes if and only if $f[x+\ii y]$ attains the minimal possible value.
\end{proof}

As far as the main text is concerned, we note that $L^2(\mathbb R)$ and $\ker \mathcal G_\pm$, 
are all Hilbert spaces with a naturally defined $\cc$-operation, and are therefore complex-$\cc$ Banach spaces.
Therefore, $\delta V_\eps[h_*, \delta h] = 0$ for all $\delta h \in \ker \mathcal G_-$ and $h_*$ 
is the necessary and sufficient condition that satisfies
\cref{Opt}, according to \cref{thrm:GateauxDifComplex}.


\section{Proof of \cref{TheoremBound}}
\label{sec:ProofTHeorem2}

The proof of \cref{TheoremBound} involves establishing the existence and square-integrability of $h(\omega)$
derived from the data $(S_{xy},S_{yy})$ transformed with $f(\omega)$.
When proven, these properties turn out sufficient to satisfy assumptions of \cref{Theorem1} on modified data.

First, we establish the existence of $h(\omega)$ for given data $(S_{xy}, S_{yy})$ that satisfies the assumptions of \cref{TheoremBound}. 
To find an optimal causal filter $h(\omega)$, we seek a minimum of $V_{\eps}[h]$ given by \cref{vare} in the space of measurable functions.
Necessarily, $h(\omega)$ has to be measurable, otherwise most of the quantities of the form of \cref{eq:See} are ill-defined in the context of integration. 

The filter $h(\omega)$ exists if and only if $\inf_h[V_\eps[h]]$ is finite, where the infimum is taken over all functions that are analytic in the upper half plane.
The function that realizes the infimum is the optimal filter.
To show that the optimal $h(\omega)$ necessarily exists, we bound $\inf_h[V_\eps[h]]$ from above by a finite number by using the properties of $(S_{xx}, S_{xy}, S_{yy})$. 
The first term in $S_{\eps \eps}(\omega)$ given by \cref{eq:See} has no dependence on $h$, and therefore the optimal filter cannot explicitly depend on $S_{xx}(\omega)$. 
Additionally, adding and subtracting a constant from $V_{\eps}[h]$ does not change the value of the optimal filter $h(\omega)$ either.
We subtract $S_{xx} - |S_{xy}|^2/S_{yy}$ from $S_{\eps\eps}$, which does not change the value of the optimal filter, and  rewrite $V_\eps[h]$ as
\begin{equation}\label{eq:AppVeps}
	V_\eps[h] = \int \frac{\dd \omega}{2 \pi} S_{yy}(\omega) \left| h^\cc(\omega) - \frac{S_{xy}(\omega)}{S_{yy}(\omega)} \right|^2.
\end{equation}
\cref{eq:AppVeps} has a finite value infimum over $h$ whenever $|S_{xy}|^2/S_{yy}$ is integrable, because 
\begin{equation}
	0 \leq \inf_h V_{\eps}[h] \leq V_{\eps}[h = 0] = \int \frac{\dd \omega}{2 \pi} \frac{|S_{xy}(\omega)|^2}{S_{yy}(\omega)} < \infty.
\end{equation}
For the data that satisfies the assumptions of \cref{TheoremBound}, $|S_{xy}|^2/S_{yy}$ is always integrable, as it is guaranteed by inequalities $2 \alpha_x - \alpha_y > 1$ and $2 \beta_x - \beta_y < 1$. 
Therefore, optimal $h(\omega)$ with a finite $V_\eps[h]$ exists for the data $(S_{xy}, S_{yy})$ that satisfies the assumptions of \cref{TheoremBound}.

With the existence of $h(\omega)$ established, we proceed to establish the asymptotics of the transformed data
$S_{xy}' = f^\cc S_{xy} , S_{yy}' = \abs{f}^2 S_{yy}$.
It is straightforward to observe that 
after the transformation by $f(\omega)$, 
with asymptotics $|f(\omega \rightarrow \infty)|\sim |\omega|^{\alpha}$ and 
$|f(\omega\rightarrow 0)| \sim |\omega|^{\beta}$, the asymptotics of the data is 
\begin{align}
	S_{yy}^\prime(\omega \rightarrow \infty) &\sim \frac{A_y}{|\omega|^{\alpha_y - 2 \alpha}},
	\\
	S_{yy}^\prime(\omega \rightarrow 0) &\sim \frac{B_y}{|\omega|^{\beta_y - 2 \beta}},
	\\
	S_{xy}^\prime(\omega \rightarrow \infty) &= O(|\omega|^{-\alpha_x + \alpha}),
	\\
	S_{xy}^\prime(\omega) &= O(|\omega|^{-\beta_x + \beta}).
\end{align} 
Choosing $\alpha = \alpha_y/2$ and $\beta = \beta_y /2$, as articulated by the main text, yields
\begin{align}
	S_{yy}^\prime(\omega \rightarrow \infty) &\sim A_y, \label{eq:SyyPrimeInf}
	\\
	S_{yy}^\prime(\omega \rightarrow 0) &\sim B_y, \label{eq:SyyPrimeZero}
	\\
	S_{xy}^\prime(\omega \rightarrow \infty) &= O(|\omega|^{-\alpha_x + \alpha_y/2}), 
	\\
	S_{xy}^\prime(\omega \rightarrow 0) &= O(|\omega|^{-\beta_x + \beta_y/2}),
\end{align} 
which are the expressions in \cref{eq:Thrm2Asymp} of the main text.

To prove the square-integrability of $S_{xy}^\prime$ and the positivity of $S_{yy}^\prime$, we scrutinize 
these functions more closely.
Since $S_{xy}$ is continuous and $f(\omega)$ is analytic in the upper half plane, $S_{xy}^\prime$ is also continuous, and therefore locally integrable on all intervals that don't include $\omega = 0, \pm \infty$. 
Since $ -\alpha_x + \alpha_y /2< -1/2$, and $- \beta_x + \beta_y/2 > - 1/2$, the function $S_{xy}^\prime$ is of a sufficiently fast fall-off at $\omega = \infty$ and sufficiently slow growth at $\omega = 0$ to be locally integrable in the neighbourhood of these points.
Therefore $S_{xy}^\prime$ is in $L^2(\mathbb R)$.
Since $S_{yy}$ is positive and finite away from $\omega = 0, \pm \infty$, and $f(\omega)$ has no zeros in the upper half-plane, the transformed data $S_{yy}^\prime$ is also positive and finite everywhere, except 
perhaps $\omega = 0, \pm \infty$. 
Since $S_{yy}^\prime$, obeys \cref{eq:SyyPrimeInf,eq:SyyPrimeZero} and $A_{y}, B_y > 0$, it follows 
that $S_{yy}^\prime$ is continuous and positive at $\omega = 0, \infty$.
As a result, $S_{yy}^\prime$ is continuous and bounded by positive numbers from above and below.

Finally, we analyse the behaviour of the optimal filter $h^\prime$ that corresponds to transformed data $(S_{xy}^\prime, S_{yy}^\prime)$. 
Substitution of $h = f h^\prime$ into \cref{eq:AppVeps} results in
\begin{equation}\label{eq:SeeSquare}
    V^\prime_{\eps}[h^\prime] \eqdef \int\limits_{-\infty}^{+\infty} \frac{\d{\omega}}{2 \pi} S_{yy}^\prime \left| h^{\prime \cc} - \frac{S_{xy}^\prime}{S_{yy}^\prime} \right|^2.
\end{equation}
As it was established above, $V^\prime_{\eps}[h^\prime] = V_\eps[h]$ is necessarily finite, and therefore the integrand of \cref{eq:SeeSquare} is integrable. 
Necessarily, $h^{\prime\cc} - S_{xy}^\prime / S_{yy}^\prime$ is square-integrable, since $S_{yy}^\prime$ is a continuous positive function bounded by a positive number from below. 
Since $S_{xy}^\prime/S_{yy}^\prime$ is also square-integrable,  $h^\prime$ itself must be square-integrable and therefore a member of $L^2(\mathbb R)$.
Integrability of $S_{yy}^\prime |h^\prime|^2$ and $S_{xy}^\prime h^{\prime\cc}$, and square-integrability of $S_{yy}^\prime h^\prime$ are simple consequences of the square-integrability of $h^\prime$.

The properties of $(S_{xy}^\prime, S_{yy}^\prime)$ and its optimal causal filter $h^\prime$ that were established above satisfy assumptions of \cref{Theorem1}.
Therefore, $h^\prime$ can be found as a solution to \cref{Gplus,Gminus} for transformed data $(S_{xy}^\prime, S_{yy}^\prime)$. 
This concludes the proof of \cref{TheoremBound}.


\section{Proof of \cref{thm:eigHilbert}}\label{app:eigHilbert}

Consider a function $g \in L^2(\mathbb R)$ and the map
\begin{equation}\label{eq:map}
	g(\omega) \mapsto \tilde g(u) = \sqrt{\pi} e^{ \ii u/2} \sec(u/2) g\left(\tg(u/2)\right),
\end{equation}
whose domain is the unit circle $\mathbb{T}$.
Under the map $g \mapsto \tilde g$, the set of eigenfunctions $\phi_n(\omega)$ in \cref{eq:Heig} gets mapped to 
$\tilde \phi_n(u) = e^{\ii nu}$. 
This map preserves the (hermitian) inner product in the sense that for $g_1, g_2 \in L^2(\mathbb R)$ 
and their images $\tilde g_1, \tilde g_2  \in L^2(\mathbb T)$,
\begin{equation}\label{normconv}
	\int_{-\infty}^{+\infty} \d \omega \, g_1^\cc(\omega) g_2(\omega) = \int_{-\pi}^{\pi} \frac{\d u}{2 \pi} \, \tilde g_1^\cc(u) \tilde g_2(u).
\end{equation}

The set $\{\tilde \phi_n\}$ for $n \in \mathbb Z$ is a Fourier basis and forms a complete orthonormal set in $L^2(\mathbb T)$. 
The properties of the Fourier basis guarantee convergence in norm of the Fourier series for a function $\tilde g(u)$:
\begin{equation}\label{eq:MappedExpansion}
	\tilde g(u) = \sum_{n = -\infty}^{+\infty} \tilde g_n \tilde \phi_n(u), \quad \tilde g_n = \int_{-\pi}^{\pi} \frac{\dd u}{2 \pi} \, \tilde g(u) \tilde \phi_n^*(u).
\end{equation}
Additionally, the series in \cref{eq:MappedExpansion} converge pointwise almost everywhere by Carleson-Hunt theorem \cite{carleson1966convergence,Hunt1968,Jorsboe1982-ef}.

Since by construction the map $g \rightarrow \tilde{g}$ is a bijection between $L^2(\mathbb R)$ and $L^2(\mathbb T)$ that also preserves the (hermitian) inner product, the series for the pre-image $g(\omega)$ also converges almost everywhere:
\begin{equation} \label{expansion}
	g(\omega) = \sum_{n = -\infty}^{+\infty} g_n \phi_n(\omega),
\end{equation}
where
\begin{equation}\label{coefs}
	g_n \equiv \int_{-\infty}^{+\infty} \d \omega \, g(\omega) \phi_n^\cc(\omega) = \int_{-\pi}^{\pi} \frac{\d u}{2\pi} \, \tilde g(u) \tilde \phi_n^\cc(u) \equiv \tilde g_n.
\end{equation}
Therefore, $\{ \phi_n \}$ is a complete orthonormal set in $L^2(\mathbb R)$, and this completes the proof of \cref{thm:eigHilbert}.


\section{Positive-definiteness of $\vec T$ and $\vec T^{(n)}$}\label{sec:Tpositive}

The uniqueness of the solution to $\vec T \vec h = \vec s$ and the positive-definiteness of the truncated Toeplitz matrix $\vec T^{(n)}$ can be proven with an identical approach. 

To establish some preliminaries, consider
an arbitrary function $g(\omega)$, and the vector $\vec g = \{ g_k = \langle \phi_k, g 
\rangle_{L^2(\mathbb R)} \}$ for $k \geq 0$.
For two functions $g_1(\omega)$ and $g_2(\omega)$, such that $g_{1,k}, g_{2,k} = 0$ for $k<0$, the Hilbert space inner product can be rewritten as $\vec g_1^\dagger \vec g_2 = \langle g_1, g_2 \rangle_{L^2(\mathbb R)}$. 
Setting $g_1(\omega) = g(\omega)$ and $g_2(\omega) = S_{yy}(\omega) g(\omega)$ gives
\begin{equation}\label{eq:AppfRel}
    \langle g, S_{yy} g \rangle_{L^2(\mathbb{R})} = \vec g^\dagger \vec T \vec g.
\end{equation}

With these preliminaries in place, we are ready to tackle the uniqueness of $\vec h$ with a proof by contradiction.
Assume there exists two distinct solutions to $\vec T \vec h = \vec s$, denoted as $\vec h_1$ and $\vec h_2$; 
then, necessarily, $\vec T(\vec h_1 - \vec h_2) = 0$. 
For $g(\omega) = h_1(\omega) - h_2(\omega)$, the relation in \cref{eq:AppfRel} becomes
\begin{equation}\label{eq:Apph1h2}
    \langle h_1 - h_2, S_{yy}(h_1 - h_2) \rangle_{L^2(\mathbb R)} 
    = (\vec h_1 - \vec h_2)^\dagger \vec T (\vec h_1 - \vec h_2) = 0.
\end{equation}
The left side of this equation can be estimated by using the fact that $\inf S_{yy}(\omega) > 0$: 
\begin{equation}
\begin{split}
    0 &= \langle h_1 - h_2, S_{yy}(h_1 - h_2) \rangle_{L^2(\mathbb R)} \\
    &= \int \dd \omega \, S_{yy}(\omega) |h_1(\omega) - h_2(\omega)|^2 \\
    &\geq \inf(S_{yy}) \int \dd \omega \, |h_1(\omega) - h_2(\omega)|^2 \geq 0.
\end{split}
\end{equation}
This immediately implies that $\int \dd \omega \, |h_1 - h_2|^2 = 0$, and therefore $h_1(\omega) - h_2(\omega) = 0$ almost everywhere. 
Consequently, it is necessarily true that $\vec h_1 = \vec h_2$ and the solution to $\vec T \vec h = \vec s$ is unique, and the causal filter $h(\omega)$ is unique almost everywhere. 
This concludes the proof of the uniqueness of $\vec h$.

The positive-definiteness of $\vec T^{(n)}$ can be proven using an analogous approach.
Consider a function $g^{(n)}(\omega)$ with a corresponding vector $\vec g^{(n)}$, such that $g_k = 0$ for all $k>n$ and $k < 0$. 
Then it is true that
\begin{equation}
    \vec g^{(n)\dagger} \vec T^{(n)} \vec g^{(n)} = \vec g^{(n)\dagger} \vec T \vec g^{(n)} = \langle g^{(n)}, S_{yy} g^{(n)}\rangle.
\end{equation}
Since $\inf S_{yy}(\omega) > 0$ due to \cref{TheoremBound}, the following estimate holds:
\begin{equation}
    \vec g^{(n)\dagger} \vec T^{(n)} \vec g^{(n)} \geq \inf(S_{yy}) \langle g^{(n)}, g^{(n)} \rangle = \inf(S_{yy}) \vec g^{(n)\dagger} \vec g^{(n)}.
\end{equation}
Since $\inf(S_{yy}) > 0$, the inequality above implies that $\vec T^{(n)}$ is positive definite matrix with 
its smallest eigenvalue not less than $\inf S_{yy}(\omega)$.

A argument also shows that the largest eigenvalue of $\vec T^{(n)}$ is not greater than $\sup S_{yy}(\omega)$.
The same conclusions about the minimal and maximal eigenvalues are true for $\vec T$ as an infinite-dimensional matrix.


\section{Membership of causal filter in $L^2(\mathbb R)$}
\label{sec:HpSpace}

\cref{Theorem1} requires that $h, S_{yy}h, S_{xy} \in L^2(\mathbb{R})$, which insures that the causal
filter $h$ is solved by the pair of operator equations [\cref{Gplus,Gminus}]: 
\begin{align}
    \mathcal{G}_+[h S_{yy}] & = \mathcal{G}_+[S_{xy}] \label{appGplus} \\
    \mathcal{G}_-[h] &= 0 \label{appGminus}.
\end{align}
\cref{TheoremBound} ensures that any data that satisfies its assumptions can be cast into new 
data $(S_{xy}', S_{yy}')$ which determines a causal filter $h'$, such that these satisfy the assumptions of \cref{Theorem1}, and solves \cref{appGplus,appGminus}.
However, since the proof of \cref{TheoremBound} in \cref{sec:ProofTHeorem2} produces the modified data through 
an argument that does not account for \cref{appGplus,appGminus}, these theorems only imply that the solution
$h$ must be sought in $L^2(\mathbb{R})$ and do not guarantee that the solution of \cref{appGplus,appGminus}
belongs to this space.

This section explicitly verifies that the solution of \cref{appGplus,appGminus} is indeed
consistent with \cref{Theorem1,TheoremBound}.
Specifically, we shall prove that if $S_{xy} \in L^2(\mathbb R)$ and $S_{yy}$ is continuous 
with $0 < \inf S_{yy} \leq \sup S_{yy} < \infty$, then there is always a unique $h \in L^2(\mathbb R)$ that satisfies \cref{appGplus,appGminus}.
The modified data of \cref{TheoremBound} satisfies the property above, and therefore allows for a 
self-consistent solution of the filter.

The proof is essentially a transcription of recent results in the theory of Toeplitz 
operators \cite{Bottcher2006-nn} for functions on the unit circle $\mathbb{T}$, via the isometric
isomorphism in \cref{eq:map}, to Toeplitz operators for functions on $\mathbb{R}$. This bridge is established
by the following lemma.
\begin{lemma}\label{lem:TildeMap}
    The map $\sim: L^2(\mathbb R) \rightarrow L^2(\mathbb T)$, affecting the transformation $g \mapsto \tilde{g}$ 
    in \cref{eq:map} is an isometric isomorphism 
    which maps the Fourier basis $\{ e^{\ii k u} \}$ on $L^2(\mathbb T)$ to the basis 
    $\{ \phi_k \}$ on $L^2(\mathbb R)$.
    Further, given the function $g \in L^2(\mathbb R)$ and $S\in L^\infty(\mathbb R)$, 
    $\widetilde{S g}  = s \tilde g$, where $s(u) = S(\tan u/2)$ and $s \in L^\infty(\mathbb T)$.
\end{lemma}
\begin{proof}
    The isometry is established in \cref{app:eigHilbert}. 
    The last property can be verified explicitly.
\end{proof}

We will now relate \cref{appGplus,appGminus} to the theory of Toeplitz operators in $\mathbb{T}$ 
\cite{Bottcher2006-nn} by introducing some machinery.
The definitions and theorems from \cite{Bottcher2006-nn} that are relevant to this section will be stated without
proof for completeness, but rephrased in our language.
As a point of departure, we introduce the Riesz projector:
\begin{definition}[\S 1.42 \cite{Bottcher2006-nn}]
    The Riesz projector $P~:~L^2(\mathbb T)~\rightarrow~L^2(\mathbb T)$ is defined, for $n\in \mathbb{Z}$, by
    \begin{equation}
        P[e^{\ii n u}] = 
            \begin{cases}
                e^{\ii n u}, &n \geq 0 \\
                0, & n < 0 
            \end{cases}
    \end{equation}
Since the set of all finite Fourier sums is dense in $L^2(\mathbb T)$, the action of $P$ on any member 
of $L^2(\mathbb T)$ is defined by linearity.
\end{definition}

\noindent The Riesz projector is an orthogonal projector, and therefore defines a subspace of $L^2(\mathbb T)$ that is also a Hilbert space:
\begin{definition}[\S 1.39 \cite{Bottcher2006-nn}]\label{DefHardy}
    The Hardy space $H^2(\mathbb T) \subset L^2(\mathbb{T})$ is defined by the subset of Fourier series 
    whose negative coefficients vanish, i.e.
    \begin{equation}
        H^2(\mathbb T) = \{ g \in L^2(\mathbb T) \mid P[g] = g \}.
    \end{equation}
\end{definition}

\noindent Below we denote $H^2(\mathbb T)$ simply as $H^2$.
Hardy space is a suitable space to define the Toeplitz operators:
\begin{definition}[\S 2.6 \cite{Bottcher2006-nn}]\label{defRiesz}\label{app:Tcf}
    Given a function $c \in L^\infty(\mathbb T)$, the Toeplitz operator of $c$ is a linear operator $T(c):H^2 \rightarrow H^2$ defined as $T(c)[f] \eqdef P[c f]$. 
\end{definition}

\noindent With the definitions laid out above, we are ready to translate the problem defined by 
\cref{appGplus,appGminus} into the language of Toeplitz operators on $H^2(\mathbb{T})$ \cite{Bottcher2006-nn}. 
First, we describe the transformation of $\mathcal G_+$ by the map $\sim$.
\begin{lemma}\label{lem:TildeOperatorMap}
    The isometry $\sim: L^2(\mathbb R) \rightarrow L^2(\mathbb T)$ induces an isometry 
    $A \mapsto \tilde{A}$ 
    on the space of bounded linear operators, defined by $\tilde A[\tilde g] = \widetilde{A[g]}$ 
    for all $g \in L^2(\mathbb R)$. 
    This induced map takes the projector $\mathcal{G}_+$ to the Riesz projector $P$.
\end{lemma}
\begin{proof}
The $\sim$ map is bijective, so $g$, the preimage of $\tilde g$, always exists and is unique.
Thus  $\tilde A$ is unambiguously defined by $\tilde A[\tilde g] = \widetilde{A[g]}$.
Since $\sim$ is bijective and linear, the map $A \mapsto \tilde A$ is also bijective and linear for all 
operators $A$; since $\sim$ preserves the inner product, the map $A \mapsto \tilde{A}$ preserves operator norm. 

It is left to prove that $\mathcal G_+$ is mapped to $P$.
Due to \cref{lem:TildeMap}, the basis $\{ \phi_k(\omega) \}$ is bijectively mapped onto the Fourier basis 
$\{ e^{\ii k u} \}$.
Explicit calculation shows that $\widetilde{\mathcal G_+[\phi_k]} = P[e^{\ii ku}]$
for all $k \in \mathbb Z$. 
Therefore $\widetilde{\mathcal{G}_+} = P$, since both $\{ \phi_k \}$ and $\{ e^{\ii ku} \}$ are bases 
in their Hilbert spaces.
\end{proof}

The following theorem establishes the equivalence between \cref{appGplus,appGminus} and the 
Toeplitz operator in \cref{app:Tcf}.
\begin{theorem}\label{thrm:ToeplitzEquivalence}
    Let $S_{xy}, h \in L^2(\mathbb R)$ and $S_{yy}(\omega)$ be continuous such that $\sup S_{yy} < \infty$.
    Then, the pair of \cref{appGplus,appGminus}
     is equivalent to the following equation for $\tilde h \in H^2$:
    \begin{equation}\label{eq:appToeplitz}
        T(s_y)[\tilde h] = P[\tilde S_{xy}],
    \end{equation}
    where $s_y(u) \eqdef S_{yy}(\tan u/2)$.
\end{theorem}
\begin{proof}
First, we establish that the inclusion $\tilde h \in H^2$ is equivalent to \cref{appGminus}. 
\cref{lem:TildeMap} implies that $\tilde h \in L^2(\mathbb T)$, so this result should be further improved.
According to \cref{lem:TildeOperatorMap}, \cref{appGminus} is mapped into $(1 - P)[\tilde h] = 0$, since $\mathcal G_- = 1 - \mathcal G_+$.
Thus by \cref{defRiesz}, $\tilde h \in H^2$. 

Next, we show that \cref{appGplus} is equivalent to \cref{eq:appToeplitz}.
Due to \cref{lem:TildeOperatorMap}, application of $\sim$ map to \cref{appGplus} yields $P[\widetilde{S_{yy} h}] = P[\tilde S_{xy}]$.
Further, the relation $\widetilde{S_{yy} h} = s_y \tilde h$ holds due to \cref{lem:TildeMap}. 
Finally, $s_y \in L^\infty(\mathbb T)$ due to $\sup s_y = \sup S_{yy} < \infty$.
As a result, $T(s_y)$ is a valid Toeplitz operator.
Both sides of \cref{eq:appToeplitz} are members of $H^2$, since $\tilde h, \tilde S_{xy} \in L^2(\mathbb T)$.
This verifies that \cref{appGplus,appGminus} are equivalent to \cref{eq:appToeplitz}. 
\end{proof}

\cref{thrm:ToeplitzEquivalence} implies that solving \cref{appGplus,appGminus} is equivalent to the 
invertibility of $T(s_y)$. 
The invertibility of linear operators requires boundedness: a bounded operator is invertible if and only if it is bijective, and its two-sided inverse is bounded.
The following lemma establishes the norm of Toeplitz operators on $H^2$.
\begin{lemma}[\S 2.8 \cite{Bottcher2006-nn}]\label{lem:ToeplitzBounded}
    The Toeplitz operator $T(c):H^2 \rightarrow H^2$ for $c \in L^\infty (\mathbb{T})$ is bounded; in
    particular the operator norm $\norm{T(c)}_{H^2} = \norm{c}_{L^\infty(\mathbb T)} < \infty$.
\end{lemma}

The following theorems further characterize the invertibility of $T(c)$ for a continuous $c$.
First, we define an auxiliary notion of the index of a continuous function.
\begin{definition}[\S 2.41 \cite{Bottcher2006-nn}]
    Given a continuous function $c:\mathbb T \rightarrow \mathbb C$, the index of $c$, denoted $\mathrm{ind}(c)$, 
    is the winding number of the contour $c(u)$ around $0$ in the complex plane for $u$ running from $-\pi$ 
    to $+\pi$.
\end{definition}

\noindent The index of a function is a useful quantity that characterizes the invertibility of Toeplitz operators.

\begin{theorem}[\S 2.42 \cite{Bottcher2006-nn}]\label{thrm:ToeplitzInvertable}
    For a continuous function $c: \mathbb T \rightarrow \mathbb C$, $T(c)$ is invertible if and only if 
    $c(u) \neq 0$ for all $u \in \mathbb T$ and $\mathrm{ind}(c)$ = 0.
\end{theorem}

\noindent We now apply these to characterize the solutions of \cref{appGplus,appGminus}.
\begin{theorem}\label{thrm:SystemOfEqUniqueness}
    Let $S_{xy} \in L^2(\mathbb R)$ and $S_{yy}:\mathbb R \rightarrow \mathbb R$ be a continuous, 
    positive, symmetric function, such that $0 < \inf S_{yy} \leq \sup S_{yy} < \infty$, and 
    $S_{yy}(\omega \rightarrow \infty) = \mathrm{const}$. 
Then, there is a unique $h \in L^2(\mathbb R)$ that solves \cref{appGplus,appGminus}.
\end{theorem}
\begin{proof}
As shown by \cref{thrm:ToeplitzEquivalence}, solving \cref{appGplus,appGminus} is equivalent to inverting 
$T(s_y)$ in \cref{eq:appToeplitz}.
The invertibility of $T(s_y)$ is confirmed by verifying that $s_y$ satisfies the assumptions of \cref{lem:ToeplitzBounded,thrm:ToeplitzInvertable}.
Assumptions of  \cref{lem:ToeplitzBounded} are satisfied, because
 $S_{yy} \in L^\infty(\mathbb R)$ implies $s_y \in L^\infty(\mathbb T)$, so $T(s_y)$ is a bounded operator.
To verify the assumptions of \cref{thrm:ToeplitzInvertable}, we check the continuity, positivity, and index of $s_y$.
Function $s_y(u)$ is continuous on $\mathbb T$ with $s_y(\pm \pi) = S_{yy}(\omega \rightarrow \infty)$, because
 $S_{yy}$ is continuous on $\mathbb R$, and $S_{yy}(\omega \rightarrow \infty)$ exists.
 The symmetry of $S_{yy}$ implies that $s_y(+\pi)=s_y(-\pi)$.
Further, $\inf s_y >0$, since $\inf S_{yy}> 0$, and therefore $s_y$ is a strictly positive function.
Immediately, positivity of $s_y$ implies that $\mathrm{ind}(s_y) = 0$ and $s_y(u) \neq 0$ for all $u \in \mathbb T$.
Therefore, $s_y$ satisfies conditions of \cref{thrm:ToeplitzInvertable} and $T(s_y)$ is invertible.
Since $T(s_y)^{-1}$ is bounded and bijective on $H^2$, $\tilde h$ is unique and is a member of $H^2$. 
Finally, a unique $\tilde h \in H^2$ implies the existence of a unique $h \in L^2(\mathbb R)$ due to \cref{lem:TildeMap}.
\end{proof}

As a closing remark, it is important to note that modified data produced by \cref{TheoremBound} obeys conditions sufficient to satisfy the assumptions of \cref{thrm:SystemOfEqUniqueness}.
Therefore, there is a unique causal filter $h$ that solves \cref{Gplus,Gminus} for a given set of modified data produced by \cref{TheoremBound}.


\section{Proof of \cref{SolutionConvergence}}\label{sec:ProofTheorem4}
The proof of \cref{SolutionConvergence} is straightforward and follows from the arguments in the main text and the information provided in \cref{app:eigHilbert,sec:Tpositive}; in the following, we explicate the argument.

As was alluded in the main text, the proof of \cref{SolutionConvergence} essentially boils down to estimating 
$\norm{h - h^{(n)}}_{L^2(\mathbb R)}$.
The two types of contributions arise, denoted $\delta h_1^{(n)}(\omega)$ and $\delta h_2^{(n)}(\omega)$. 
The contribution $\delta h_1^{(n)}(\omega)$ corresponds to the truncation of the sum 
$h(\omega) = \sum_{k = 0}^{\infty} h_k \phi_k(\omega)$ at finite $n$. 
The contribution $\delta h_2^{(n)}(\omega)$ corresponds to a discrepancy between $h_k$ and $h_k^{(n)}$ for $k \leq n$.
The expressions for these errors in terms of $h_k$ and $h_k{(n)}$ are
\begin{align}\label{eq:deltaH1}
	\delta h_1^{(n)}(\omega) &= \sum_{k = n+1}^{\infty} h_k \phi_k(\omega)
	\\
	\delta h_2^{(n)}(\omega) &= \sum_{k = 0}^{n} (h_k - h_k^{(n)}) \phi_k(\omega).
\end{align}
Note that both $\delta h_1^{(n)}(\omega)$ and $\delta h_2^{(n)}(\omega)$ are necessarily $L^2(\mathbb R)$ functions, since $h, h^{(n)} \in L^2(\mathbb R)$.

It was established in \cref{thm:eigHilbert} and Supplemental Information \cref{app:eigHilbert} that the map 
$g \mapsto \tilde g$ is bijective and preserves the inner product in $L^2(\mathbb R)$ and $L^2(\mathbb T)$. 
Therefore, we choose to work with $\tilde h_1^{(n)}(u)$ and $\tilde h_2^{(n)}(u)$ (members $L^2(\mathbb T)$) for convenience.
Explicit expressions for  $\delta \tilde h_1^{(n)}(u)$ and $\delta \tilde h_2^{(n)}(u)$ are
\begin{align}\label{eq:deltah1}
	\delta \tilde h_1^{(n)}(u) &= \sum_{k = n+1}^{\infty} h_k e^{\ii k u}
	\\\label{eq:deltah2}
	\delta \tilde h_2^{(n)}(u) &= \sum_{k = 0}^{n} (h_k - h_k^{(n)}) e^{\ii k u}.
\end{align}
For the same reasons as above, $\tilde h$, $\delta \tilde h_1^{(n)}$, and $\delta \tilde h_2^{(n)}(u)$ 
are all members of $L^2(\mathbb T)$.

First, we estimate the $L^2(\mathbb T)$ norm of $\delta \tilde h_1^{(n)}$.
Parseval's identity allows to evaluate the $L^2(\mathbb T)$ norm of $\tilde h$ as
\begin{equation}\label{eq:ParsevalH}
	||\tilde h||_{L^2(\mathbb T)} = \int_{-\pi}^{\pi} \frac{\dd u}{2 \pi} \; |\tilde h(u)|^2 = \sum_{k=0}^{\infty} |h_k|^2.
\end{equation}
From the convergence of the series in the right side of \cref{eq:ParsevalH}, it immediately follows that
\begin{equation}\label{eq:Parsevaldeltah1}
	\sum_{k=n}^{\infty} |h_k|^2 \rightarrow 0, \quad n \rightarrow \infty.
\end{equation}
Consequent application of Parseval's identity to \cref{eq:Parsevaldeltah1} results in 
\begin{equation}
	||\delta \tilde h_1^{(n)}||_{L^2(\mathbb T)} = \int_{-\pi}^\pi \frac{\d u}{2 \pi} \, |\delta \tilde h_1^{(n)}(u)|^2 \rightarrow 0, \quad n \rightarrow \infty.
\end{equation}
This result can be restated as
\begin{equation}
	\delta h_1^{(n)}(\omega) \rightarrow 0, \quad n \rightarrow \infty,
\end{equation}
where convergence is understood in the sense of $L^2(\mathbb R)$ norm, since the $\sim$ map preserves the norm.

We proceed to constrain the norm of $\delta h_2^{(n)}(\omega)$.
\cref{eq:linearEq} of the main text can be rewritten as
\begin{equation}\label{eq:ExactSystem}
	\sum_{k = 0}^{n} t_{l-k} h_k  = S_{xy, l} - \delta S_l^{(n)}, \quad \delta S_l^{(n)} =  \sum_{k = n+1}^{\infty} t_{l - k} h_k
\end{equation}
where $0\leq l \leq n$. 
The quantity $\delta S_l^{(n)}$ characterizes the error induced in $S_{xy, l}$ by the finite size truncation of the system of equations.
We proceed first to estimate the magnitude of $\delta S_l^{(n)}$ and then estimate its impact on the norm of $h_k^{(n)}$, which is exactly the effect of $\delta h_2^{(n)}$.
To characterize $\delta S_l^{(n)}$, we put the estimate on $\sum_{k = n+1}^\infty t_{l-k}h_k$, which describes the coordinates of $\vec T$ acting on the corresponding vector of $\delta h_2^{(n)}(\omega)$.
In Supplementary Information \cref{sec:Tpositive} we argued that $\vec T$, as an infinite matrix, has a maximal eigenvalue bounded by $\sup [S_{yy}]$.
This immediately allows to estimate the norm of $\delta S_l^{(n)}$ as
\begin{multline}
	\sum_{l = 0}^{n} \left| \delta S_l^{(n)} \right|^2 \leq \sum_{l = 0}^{\infty} \left|\delta S_l ^{(n)} \right|^2 \leq 
	\\
	\leq (\sup[S_{yy}])^2 \sum_{k = n+1}^{\infty} |h_k|^2 \rightarrow 0, \quad n \rightarrow \infty.
\end{multline}
Therefore, the second term in the RHS of \cref{eq:ExactSystem} vanishes by norm with the growth of $n$.

With this estimation on the corrections to $S_{xy, l}$ in hand, we estimate $h_k - h_k^{(n)}$ by inverting $\vec T^{(n)}$ in \cref{eq:ExactSystem}.
It was demonstrated in \cref{sec:Tpositive} that $\vec T^{(n)}$ has its minimal eigenvalue bounded from below by $\inf [S_{yy}]$.
The following estimation follows:
\begin{equation}\label{eq:estdeltah2app}
	\sum_{k = 0}^{n} |h_k - h_k^{(n)}|^2 \leq (\inf [S_{yy}])^{-2} \sum_{l = 0}^{\infty} |\delta S_l^{(n)}|^2 \rightarrow 0,
\end{equation}
as $n \rightarrow \infty$. 
Application of Parseval's identity to \cref{eq:deltah2} in combination with \cref{eq:estdeltah2app} leads to
\begin{equation}
	||\delta \tilde h_2^{(n)}||_{L^2(\mathbb T)} = \int_{-\pi}^{\pi} \d u \, |\delta \tilde h_2^{(n)}(u)|^2 \rightarrow 0, \quad n \rightarrow \infty,
\end{equation}
and therefore 
\begin{equation}
	\delta h_2^{(n)}(\omega) \rightarrow 0, \quad n \rightarrow \infty
\end{equation}
in the sense of norm, since $\sim$ map preserves the inner product.  

Since we have shown that both $\delta h_1^{(n)}$ and $\delta h_2^{(n)}$ vanish in norm, 
\begin{equation}\label{eq:WeakPointwiseConv}
	h^{(n)}(\omega) \rightarrow h(\omega), \quad n \rightarrow \infty
\end{equation}
in norm.
As an immediate consequence of \cref{eq:WeakPointwiseConv}:
\begin{equation}
	h_k^{(n)} \rightarrow h_k, \quad n \rightarrow \infty
\end{equation}
for each fixed $k$.
This concludes the proof of \cref{SolutionConvergence}

\section{Proof of \cref{trm:WienerAndSmoothWiener}}
\label{sec:WienerAndSmoothWiener}

In \cref{sec:HpSpace} we translated the problem of solving \cref{Gplus,Gminus} into the language of Toeplitz operator theory \cite{Bottcher2006-nn}.
We consequently showed that the optimal causal filter $h(\omega)$ is unique for the data produced by \cref{TheoremBound}.
We finally demonstrated that
\cref{SolutionConvergence} certifies that the numerical approximation $h^{(n)}(\omega)$ converges to $h(\omega)$ in norm.
The intent of \cref{trm:WienerAndSmoothWiener} is a strengthening of the conclusion of 
\cref{SolutionConvergence} for additional constraints on $(S_{xy}, S_{yy})$.

To construct a rigorous proof of \cref{trm:WienerAndSmoothWiener}, we capitalize on the developments in  Supplementary Information \cref{sec:HpSpace,sec:ProofTheorem4}.
\cref{trm:WienerAndSmoothWiener} results from the following pair of theorems.

\begin{theorem}\label{thrm:WienerPartTheorem}
	Let $S_{xy} \in L^2(\mathbb R)$ and $S_{yy}:\mathbb R \rightarrow \mathbb R$ be a positive, symmetric, continuous function, such that $0 < \inf S_{yy} \leq \sup S_{yy} < \infty$.
In addition, suppose that $\sum_{n = 0}^{\infty} |t_n|$ and $\sum_{n = 0}^{\infty} |S_{xy, n}|$ are both finite, 
where $2\pi t_k = \int_{- \pi}^{\pi} \d{u} e^{i k u} S_{yy}\left(\tan u/2\right)$ and $S_{xy, n} = \langle \phi_n, S_{xy}\rangle$.
Then,
(i) the expansion $h(\omega) = \sum_{k\geq 0}h_k \phi_k (\omega)$ converges absolutely and uniformly; 
(ii) the truncated filter $h^{(n)}(\omega)$ converges to $h(\omega)$ everywhere uniformly; 
and, (iii) the coefficients $h_k^{(n)}$ converge to $h_k$ uniformly in $k$ as $n \rightarrow \infty$.
\end{theorem}

\begin{theorem}\label{thrm:SmoothWienerPartTheorem}
Let $S_{xy} \in L^2(\mathbb R)$ and $S_{yy}:\mathbb R \rightarrow \mathbb R$ be a positive symmetric continuous function, such that $0 < \inf S_{yy} \leq \sup S_{yy} < \infty$.
Suppose there exists such $\omega_0 > 0$, such that the data are of a H\"{o}lder class $C^{m, \alpha}$ for integer $m \geq 1 $ and $0<\alpha <1$ on $[-\omega_0, \omega_0]$. Additionally suppose that $S_{yy}(1/\omega)$ and $(1 + \ii/\omega)S_{yy}(1/\omega))$ are of a H\"{o}lder class $C^{m, \alpha}$ on $[-\omega_0^{-1}, \omega_0^{-1}]$.
Then, (i) $h(\omega)$ is also of H\"{o}lder class $C^{m, \alpha}$ on $[- \omega_0, \omega_0]$ and $h(1/\omega)$ is of a H\"older class $C^{m, \alpha}$ on $[-\omega_0^{-1}, \omega_0^{-1}]$; 
(ii) $h_k = O(1/k^{m+\alpha})$;
(iii) the convergence of the truncated filter is controlled as 
$|h^{(n)}(\omega) - h(\omega)| < A n^{1-m-\alpha} $ for some $A >0$; and,
(iv) the convergence of the expansion coefficients is uniform in $k$ and controlled by 
$|h^{(n)}_k - h_k| < B n^{1-m-\alpha} $ for some $B>0$. 
\end{theorem}
During the proof, it will become evident that the assumptions of \cref{thrm:SmoothWienerPartTheorem} are strictly stronger than those of \cref{thrm:WienerPartTheorem}.
Thus, proving these two theorems is equivalent to proving \cref{trm:WienerAndSmoothWiener}, so we will examine them separately.

The proofs of \cref{thrm:WienerPartTheorem,thrm:SmoothWienerPartTheorem} are structured similarly and involve three steps:
\begin{enumerate}[nosep]
\item[I.] Establish an equivalence between the functions in $L^2(\mathbb R)$ that satisfy the assumptions of the theorem and a certain ``smooth class'' of functions on $\mathbb T$ under the map $\sim$.
\item[II.] Show that the Toeplitz operator equation in \cref{eq:appToeplitz} is invertible inside the smooth class.
\item[III.] Establish the convergence properties of the truncated filter to the exact solution of the Toeplitz operator equation.
\end{enumerate}
To execute the proof, a precise notion of ``smooth class" is necessary, which calls for some more
preparatory work.

\subsection{Toeplitz operators on Banach algebras}

We begin by listing some aspects of the modern theory of Toeplitz operators \cite{Bottcher2006-nn}. 
A perfect place to start is the notion of a Banach algebra, which recognizes the elementary notion 
of multiplication of functions and formalizes it in the context of function spaces (which are otherwise
linear vector spaces with only an addition defined on functions).

\begin{definition}\label{def:Banach}
	A Banach algebra $\mathbb A$ is a Cauchy-complete Banach space that is equipped with an additional operation called element multiplication $(a, b) \rightarrow ab \in \mathbb A$ for all $a, b \in \mathbb A$.
The multiplication map respects the norm $||\bullet ||_{\mathbb A}$ of the Banach space, requiring $||ab||_{\mathbb A} \leq ||a||_{\mathbb A} ||b||_{\mathbb A}$ for all $a, b \in \mathbb A$. 
For the purpose of this work, $\mathbb A$ is always equipped with an identity element $e \in \mathbb A$, such that $||e||_{\mathbb A} = 1$ and $e a = a e = a$ for all $a \in \mathbb A$.
\end{definition}

A simple consequence of any Banach algebra being closed under addition and multiplication is:
\begin{lemma}
	For any Banach Algebra $\mathbb A$, the exponential map $\exp[a] \eqdef \sum_{n = 0}^{\infty} a^n/n!$ is also a member of $\mathbb A$ for all $a \in \mathbb A$, and 
    $||\exp[a]||_A \leq e^{||a||_{\mathbb A}}$.
\end{lemma}
It is often convenient to talk about a Banach algebra $\mathbb B$ being continuously embedded into a 
Banach algebra $\mathbb A$:
\begin{definition}
	The Banach algebra $\mathbb B$ is said to be continuously embedded into a Banach algebra $\mathbb A$, 
    if $\mathbb B \subset \mathbb A$ and $||b||_{\mathbb A} \leq ||b ||_{\mathbb B}$ for all $b \in \mathbb B$.
\end{definition}

Relevant examples of Banach algebras are spaces $L^p(\mathbb T)$ for $1 \leq p \leq \infty$, with their corresponding norms $||\bullet ||_{L^p(\mathbb T)}$ \cite[\S 1.36]{Bottcher2006-nn}, and the
Hardy space
\begin{equation}
    H^\infty = \{ g \in L^\infty(\mathbb T): P[g] = g \}.
\end{equation}
Toeplitz operators of the form $T(a)$ can be associated with any $a$ in a Banach algebra; the advantage of 
working in this setting is that a factorization of $a$ makes sense, and moreover, the Toeplitz operator
associated with a product of elements can be related to products of their Toeplitz operators:

\begin{theorem}[\S 2.14 \cite{Bottcher2006-nn}]\label{thrm:Factorization}
	For all $a_+,b_+ \in H^\infty$ and all $c \in L^\infty(\mathbb T)$, it is true 
    that $T(a_- c b_+) = T(a_-) T(c)T(b_+)$, where $a_- \eqdef a_+^*$. 
\end{theorem}

We often use the subscript ``$+$'' to denote elements of $H^\infty$, meaning that they have no negative-index Fourier coefficients.
It's complex conjugate, denoted with the ``$-$" subscript will then have no positive-index Fourier coefficients.

Note that the order of operators is important, since different Toeplitz operators generally don't commute.
\cref{thrm:Factorization} only holds due to a relation between Toeplitz operators and $H^\infty$: for $b_+$ by the assumptions of the theorem, all negative Fourier coefficients vanish, and for $a_-$ all positive Fourier coefficients vanish.
\begin{theorem}[\S 2.18 \cite{Bottcher2006-nn}]\label{thrm:Winther}
	For all $h \in H^\infty$ such that $h^{-1} \in H^\infty$, $T(h)$ is invertible and $T(h)^{-1} = T(h^{-1})$.
\end{theorem}

\begin{theorem}\label{thrm:ConjWinther}
	For all $h^\cc \in H^\infty$ such that $(h^{-1})^\cc \in H^\infty$, $T(h)$ is invertible and $T(h)^{-1} = T(h^{-1})$.
\end{theorem}
\begin{proof}
    Since $h^\cc \in H^\infty$, $1 = T(1) = T(h h^{-1}) = T(h)T(h^{-1})$ according to \cref{thrm:Factorization}.
    Since ${h^{-1}}^\cc \in H^\infty$, $T(1) = T( h^{-1} h) = T(h^{-1}) T(h)$ according to 
    \cref{thrm:Factorization}. Therefore, $T(h^{-1}) = T(h)^{-1}$ as a two-sided inverse.
\end{proof}

We are interested in the properties of Banach algebras in the context of the theory of Toeplitz operators. 
The central object of the latter is the Riesz projector $P$, which we would like to now view as acting
on a Banach algebra $\mathbb{A}$, $P: \mathbb A \rightarrow  \mathbb A$, as a well-defined bounded operator.
Moreover, we focus on the algebras that are embedded into $L^\infty(\mathbb T)$,
since the Toeplitz operator $T(c)$ of most interest to us is defined for $c \in L^\infty(\mathbb T)$.
The following definition captures these hopes.
\begin{definition}[\S 10.15 \cite{Bottcher2006-nn}]\label{def:Decomposing}
A Banach algebra $\mathbb A$ is called a decomposing algebra, if and only if
\begin{enumerate}
\item $\mathbb A$ is a Banach algebra that is continuously embedded into $L^\infty(\mathbb T)$.
\item $\mathbb A$ contains all finite Fourier sums.
\item For all $a \in \mathbb  A$, $P[a]$ is also a member of $\mathbb A$. 
\end{enumerate}
These properties are sufficient for $P$ to be bounded in $\mathbb A$.
\end{definition}
It immediately follows from \cref{def:Decomposing} that $T(a)$ is bounded on $\mathbb A$, if $a \in \mathbb A$ and $\mathbb A$ is a decomposing algebra.
Note that the formulation of \cref{def:Decomposing} formally differs from \cite[Section 10.15]{Bottcher2006-nn} only due to the insignificant difference in the definition of the domain $\mathbb T$.

It is important to note that  $\mathbb A \cap H^\infty$ and its complex-conjugate counterpart are also a Banach algebras for a decomposing algebra $\mathbb A$,  leading to:
\begin{lemma}\label{lem:HardyExp}
Let $\mathbb A$ be a decomposing algebra.
Then for all $a, b \in \mathbb A \cap H^\infty$, $a + b$, $ab$ and $\exp[a]$ are also members of $\mathbb A \cap H^\infty$.
\end{lemma}
The only slightly non-trivial part of the proof is showing that $ab \in \mathbb A \cap H^\infty$, which follows from the convolution theorem of the Fourier transform.

As it was mentioned above, the key idea of Part II of the proof is to show that a Toeplitz operator in \cref{eq:appToeplitz} can be inverted inside the algebra $\mathbb A$. 
The properties of decomposing algebra $\mathbb A$ are insufficient to guarantee that $T(a)^{-1}$ is bounded on $\mathbb A$, even though $||T(a)||_{\mathbb A} < \infty$ for all $a \in \mathbb A$.
Consequently, we require an object with stronger properties:
\begin{definition}\label{def:InversionClosed}
	A decomposing algebra $\mathbb  A$ is called inverse-closed, if and only if for all $a \in \mathbb A$ such that $T(a)$ is invertible in $H^2$ , there are $a_+, a_- \in \mathbb A$, such that $a = a_+ a_-$ and  $a_+, a_+^{-1},a_-^\cc, (a_-^{-1})^\cc  \in \mathbb A \cap H^\infty$.
\end{definition}
\noindent The inverse-closed algebras live up to their name due to the following lemma.

\begin{lemma}\label{lem:InversionLemma}
	Let $\mathbb A$ be an inverse-closed algebra and for any $a \in \mathbb{A}$, let 
    $T(a)$ be an invertible Toeplitz operator on $H^2$. 
    Then $T(a)$ is also invertible on $\mathbb A \cap H^\infty$, and so 
    $T(a)^{-1}$ is bounded on $\mathbb A \cap H^\infty$.
\end{lemma}
\begin{proof}
Given an inverse-closed algebra $\mathbb A$ and an invertible on $H^2$ Toeplitz operator $T(a)$ for $a \in \mathbb A$, there exist
$a_+, a_- \in \mathbb A$, such that $a = a_+ a_-$ according to \cref{def:InversionClosed}.
Due to \cref{thrm:Factorization}, $T(a) = T(a_-) T(a_+)$, since $a_+, a_-^\cc \in \mathbb A \cap H^\infty$.
Operators $T(a_+)$ and $T(a_-)$ are invertible and $T(a_+)^{-1} = T(a_+^{-1})$, $T(a_-)^{-1} = T(a_-^{-1})$ according to \cref{thrm:Winther,thrm:ConjWinther}.
Since $T(a)$, $T(a_-)$, and $T(a_+)$ are all invertible, $T(a)^{-1} = T(a_+)^{-1} T(a_-)^{-1} = T(a_+^{-1}) T(a_-^{-1})$.
We expressed the inverse of $T(a)$ as a finite composition of Toeplitz operators that are bounded on $\mathbb A \cap H^\infty$. 
Therefore, $T(a)$ is invertible $\mathbb A \cap H^\infty$ and its inverse is bounded.  
\end{proof}

Finally, we define two Banach algebras that will be used in the proof of \cref{thrm:WienerPartTheorem,thrm:SmoothWienerPartTheorem}.
These are the Wiener algebra $\mathbb W$ and the algebra of positive H\"older class functions 
$C^\alpha(\mathbb T)$.

\begin{definition}[\S 1.49 \cite{Bottcher2006-nn}]
	A Wiener algebra $\mathbb W$ is a subset of $L^2(\mathbb T)$ of functions that have absolutely convergent Fourier series. It is a Banach algebra when equipped with the norm
    \begin{equation}
    	||f||_{\mathbb W} \equiv \sum_{n = -\infty}^{\infty} |f_n| < \infty,
    \end{equation}
    where $f_n$ are the Fourier coefficients of $f$.
\end{definition}
It is easy to show that $\mathbb W$ is a decomposing algebra.
Since the Fourier series of all functions in $\mathbb W$ converge absolutely, they also converge everywhere and uniformly  on $\mathbb T$.
Due to the absolute uniform convergence of the Fourier transform, all members of $\mathbb W$ are uniformly continuous.
The following theorem is known in classic literature as two separate statements: the existence of the 
Wiener-Hopf factorization and Wiener's lemma \cite{Wien32} (that if $f$ is non-zero and has an absolutely convergent Fourier series, then so does $1/f$). Indeed, Gelfand recognized Banach algebras as the natural setting for 
Wiener's lemma \cite{Gelf39}.

\begin{theorem}[\S 10.3, 10.22.a \cite{Bottcher2006-nn}]\label{eq:Wiener}
	The Wiener algebra $\mathbb W$ is a decomposing and inverse-closed algebra.
\end{theorem}

We conclude the exposition with the H\"older classes and their useful properties below.

\begin{definition}[Section 1.50, \cite{Bottcher2006-nn}, Chapter 1 \cite{Ladyenskaja1968LinearAQ}]
	H\"older class $C^{0, \alpha}(\mathbb T)$ for $0 < \alpha < 1$ is defined as a subset of $L^2(\mathbb T)$, such 
    that any $f \in C^{0,\alpha}(\mathbb T)$ satisfies
\begin{equation}
	|f(u) - f(u^\prime)| < A (|u - u^\prime| \; \mathrm{mod} \; 2 \pi)^{\alpha}
\end{equation}
for some $A>0$ and all $u, u^\prime \in \mathbb T$. 
The $\mathrm{mod}\; 2\pi$ for the absolute value on the right is understood in the sense of $|u - u^\prime | \; \mathrm{mod} \; 2 \pi = \min_{n \in \mathbb Z}(|u - u^\prime + 2 \pi n|)$. 
For integer $n \geq 0$ and $0< \alpha < 1$, the H\"older class $C^{n, \alpha}(\mathbb T)$ consists of functions $f$ such that $f^{(n)} \in C^{, 0, \alpha}(\mathbb T)$.
When equipped with the norm
\begin{equation}
	||f||_{C^{n, \alpha}(\mathbb T)} = \sum_{k = 0}^{n} ||f^{(k)}||_{L^\infty} + \sup_{x \neq y} \frac{|f^{(n)}(x) - f(y)^{(n)}|}{(|x-y| \; \mathrm{mod} \; 2 \pi)^{\alpha}},
\end{equation}
each class $C^{n,\alpha}(\mathbb T)$ is a Banach algebra  (see \cite[Eq. 1.4]{Ladyenskaja1968LinearAQ}).
\end{definition}

\begin{lemma}[\S 2.4 \cite{Zygmund_2003}]\label{lem:CProps}
	Assume that $m \geq 0$ is an integer, $0<\beta < 1$, and $\alpha = m+ \beta$. Then $f \in C^{m, \beta}(\mathbb T)$ implies that $f_n = O(n^{-\alpha})$ and $|f(u) - \sum_{k = -n}^n f_k e^{\ii u k}| = O(n^{-\alpha} \log n)$, where $f_n$ is the $n$'th Fourier coefficient of $f$.
	For all integer $m\geq 1$ and  $0< \beta < 1$, $C^{m, \beta}(\mathbb T) \subset \mathbb W$.
\end{lemma}

The following property characterizes the inverses of $T(a)$ for $a \in C^{n, \alpha}$.

\begin{theorem}\label{thrm:HolderInverseClosed}
	The H\"older class $C^{n, \alpha}(\mathbb T)$ is a decomposing algebra for all integer $n \geq 0$ and $0 < \alpha < 1$.
	The H\"older class $C^{n, \alpha}(\mathbb T)$ is an inversion-closed algebra for all integer $n \geq 1$ and $0 < \alpha < 1$.
\end{theorem}
\begin{proof}
We begin the proof by showing that $C^{n, \alpha}(\mathbb T)$ is a decomposing algebra for all integer $n \geq 0$ and $0< \alpha < 1$.
The first requirement of \cref{def:Decomposing} is satistied: indeed, $C^{n,\alpha}(\mathbb T)$ is continuously embedded into $L^\infty(\mathbb T)$, since $C^{n,\alpha}(\mathbb T) \subset L^\infty(\mathbb T)$ and $||\bullet ||_{C^{n,\alpha}(\mathbb T)} \geq ||\bullet ||_{L^\infty(\mathbb T)}$.
The second requirement of \cref{def:Decomposing} is also naturally satisfied, as all the finite Fourier Series are smooth functions. 
To check the third requirement of \cref{def:Decomposing}, we use the integral form of the Riesz projector. 
As was discussed in the main text, $\mathcal G_+ = (I + i \mathcal H)/2$, where $\mathcal H[f]$ is a Hilbert transform of $f \in L^2 (\mathbb R)$. 
Performing the $\sim$ map on both sides of $\mathcal G_+ = (I + i \mathcal H)/2$ leads to the integral representation of the Riesz projector on $C^{n,\alpha}(\mathbb T) \subset L^2(\mathbb T)$:
\begin{equation}\label{eq:SinForm}
	P[a](u) = \frac{1}{2} \left(f(u) + \int_{-\pi}^{\pi} \frac{\d u^\prime}{2\pi i} \frac{a(u^\prime + u) e^{- \ii \frac{u^\prime}{2}}}{\sin\left( \frac{u^\prime}{2} \right)}\right)
\end{equation}
 for $a \in C^{n,\alpha}(\mathbb T)$. 
The integral divergence in \cref{eq:SinForm} is understood in the sense of principal value.
The property $P[a] \in C^{0,\alpha}(\mathbb T)$ for all $a \in C^{0,\alpha}(\mathbb T)$ and $0 < \alpha < 1$ can be confirmed by substitution $z = e^{\ii u}$.
The variable substitution reduces the integral in \cref{eq:SinForm} to a Cauchy integral on a unit circle, which is known to preserve the H\"older class for $n = 0$ and $0 < \alpha < 1$ \cite[Chapter 5.1]{gakhov1966boundary}. 
From the form of \cref{eq:SinForm} it is evident that $P$ commutes with the derivative almost everywhere on $\mathbb T$.
Therefore, if $f \in C^{n, \alpha}(\mathbb T)$ with $0 < \alpha < 1$, then $\p^{m}P[f] = P[\p^m f]$ for $0 \leq m \leq n$.
Necessarily, derivative of $f$ is such that $\p^n f \in C^{0, \alpha}(\mathbb T)$ for $f \in C^{n, \alpha}(\mathbb T)$.
Therefore, $P[\p^n f] \in C^{0, \alpha}(\mathbb T)$, since $P$ preserves the H\"older class for $0 < \alpha < 1$ and $n = 0$.
As a result, $\p^n P[ f] \in C^{0, \alpha}(\mathbb T)$ and $P[f] \in C^{n, \alpha}(\mathbb T)$, since $P$ commutes with the derivative.
This shows that the requirement (3) in \cref{def:Decomposing} is satisfied and $C^{n , \alpha}(\mathbb T)$ is a decomposing algebra for integer $n \geq 0$ and $0< \alpha < 1$.

We proceed to prove that $C^{n, \alpha}(\mathbb T)$ is an inversion-closed algebra for all integer $n \geq 1$ and $0< \alpha < 1$.
Since all members of $C^{n, \alpha}(\mathbb T)$ are continuous, it follows that $T(a)$ being invertible in $H^2$ is equivalent to $a(t) \neq 0$ for all $t \in \mathbb T$ and $\mathrm{ind}(a) = 0$  (\cref{thrm:ToeplitzInvertable}).
Function $a(t) \neq 0$ is continuous on a compact $\mathbb T$, so $0 < \inf |a| \leq \sup |a| < \infty$. 
The contour $a(t)$ does not complete a single turn around $0$ when $t$ runs from $-\pi$ to $\pi$, since $\mathrm{ind}(a) = 0$, and therefore $\ln a(t)$ is a well-defined single-valued function on $\mathbb T$. 
Since logarithm is a smooth function on $[\inf |a|, \, \sup|a|] \subset (0, \infty)$, $\ln a(t)$ is necessarily the same H\"older class $C^{n, \alpha}(\mathbb T)$ as $a$. 
Since $C^{n, \alpha}(\mathbb T)$ is a decomposing algebra, $P$ is bounded on $C^{n,\alpha}$, so $b+_ = P[\ln a]$ is a member of $C^{n,\alpha} \cap H^\infty$. 
Similarly, $b_- = (1-P)[\ln a]$ is such that $b_-^\cc$ is a member of $C^{n,\alpha} \cap H^\infty$. 
Since $C^{n, \alpha}(\mathbb T)$ is a decomposing algebra, $a_+ = e^{b_+}$ and $a_- = e^{b_-}$ are such that $a_+, a_-^\cc \in C^{n, \alpha} \cap H^\infty$ according to \cref{lem:HardyExp}. 
Moreover $a_+^{-1}, (a_-^{-1})^\cc \in C^{n,\alpha}(\mathbb T)$, since $0< \inf |a_+|, \inf |a_-|$ due to $a(t) = a_+(t) a_-(t) \neq 0$ for all $t \in \mathbb T$.
On the other side, $C^{n, \alpha}(\mathbb T) \subset \mathbb W$ for $n \geq 1$ and $0 < \alpha < 0$, so
$a_+^{-1}, (a_-^{-1})^{\cc} \in \mathbb W \cap H^\infty$, according to \cref{eq:Wiener}.
As a result, $a_+^{-1}, (a_-^{-1})^{\cc} \in C^{n,\alpha}(\mathbb T) \cap \mathbb W \cap H^\infty = C^{n,\alpha}(\mathbb T) \cap H^\infty$.
\end{proof}

With these preliminaries, we are ready 
to prove \cref{thrm:WienerPartTheorem,thrm:SmoothWienerPartTheorem}.

\subsection{Proof of \cref{thrm:WienerPartTheorem}}\label{sec:WienerAlgebra}
First, we show that the functions $S_{xy}$ and $S_{yy}$ that satisfy the assumptions of \cref{thrm:WienerPartTheorem} are mapped into the Wiener algebra $\mathbb W$ by the map $\sim$.
The proof of this fact can be formulated with the following lemma:
\begin{lemma}\label{lem:TildeMapWiener}
Let $S_{xy} \in L^2(\mathbb R)$ and $S_{yy} \in L^\infty(\mathbb R)$.
In addition, suppose that $\sum_{n = 0}^{\infty} |t_n|$ and $\sum_{n = 0}^{\infty} |S_{xy, n}|$ are both finite, 
where $2\pi t_k = \int_{- \pi}^{\pi} \d{u} e^{i k u} S_{yy}\left(\tan u/2\right)$ and $S_{xy, n} = \langle \phi_n, S_{xy}\rangle$.
Then and only then $s_y(u) = S_{yy}(\tan u /2)$ and $P[\tilde S_{xy}]$ belong to $\mathbb W$.
\end{lemma}
\begin{proof}
Since $\sum_{n  = 0}^{\infty}|t_n| < \infty$ and $t_n = t_{-n}^\cc$, it is true that $\sum_{n = -\infty}^{\infty} |t_n| < \infty$.
By definition, $t_n$ is the $n$'th Fourier coefficient of $s_y$, so $s_y$ has an absolutely convergent Fourier series, leading to $s_y \in \mathbb W$.
The inverse is also true: $S_{yy} \in L^\infty(\mathbb R)$ due to $s_y \in L^\infty(\mathbb T)$, and $\sum_{n = 0}^\infty|t_n| < \infty$ due to $s_y \in \mathbb W$.
Since the $\sim$ map transforms the basis $\{ \phi_k\}$ into the Fourier basis on $\mathbb T$ (\cref{lem:TildeMap}), $S_{xy, l}$ are the Fourier coefficients of $\tilde S_{xy}$.
Then $\sum_{n =0}^{\infty} |S_{xy, l}| < \infty$ is a sum of absolute values of all the Fourier coefficients of $P[\tilde S_{xy}]$, leading to $P[\tilde S_{xy}] \in \mathbb W$. 
The inverse is true with the logic of the proof reversed.
\end{proof}

With the connection between the modified data $(S_{xy}, S_{yy})$ and the Wiener algebra established, we immediately claim the uniqueness of the corresponding filter $h$ and verify its smoothness.
Concretely, we seek to establish that a solution to $T(s_y)[h] = P[\tilde S_{xy}]$ is a member of $\mathbb W$, when $P[\tilde S_{xy}]$ and $s_y$ are in $\mathbb W$.
The algebra $\mathbb W$ is an inversion-closed algebra according to \cref{eq:Wiener} and $\inf s_y > 0$, therefore
$\tilde h = T(s_y)^{-1}[P[\tilde S_{xy}]] \in \mathbb W \cap H^\infty$ according to \cref{lem:InversionLemma}.  
In combination with \cref{lem:TildeMapWiener}, $\tilde h \in \mathbb W \cap H^\infty$ is equivalent to the statement (i) of \cref{thrm:WienerPartTheorem}.

To prove statements (ii) and (iii) of \cref{thrm:WienerPartTheorem}, we utilize the notation used in \cref{sec:ProofTheorem4}. 
The proof will be complete, if we show that $\delta h_1^{(n)}(\omega)$ and $\delta h_2^{(n)}(\omega)$, defined by \cref{eq:deltah1,eq:deltah2}, vanish uniformly with $n \rightarrow \infty$.
First, we note that $\delta \tilde h_1^{(n)}$ is a member of $\mathbb W$, so
\begin{equation}
	|\delta \tilde h_1^{(n)}(u)| \leq \sum_{k = n}^{\infty} |h_n| \rightarrow 0,
\end{equation}	
as $n \rightarrow \infty$. 
Therefore, $\delta \tilde h_1^{(n)}(u) \rightarrow 0$ everywhere and uniformly.

To estimate $\delta \tilde h_2^{(n)}(u)$, we first note that  the Wiener algebra is decomposing and therefore $||T(s_y)||_{\mathbb W} < \infty$ due to \cref{eq:Wiener}.
Then, the coefficeints $\delta S_l^{(n)}$, defined by \cref{eq:ExactSystem}, can be estimated as
\begin{equation}
	\sum_{l = 0}^{\infty} |\delta S_l^{(n)}| \leq ||T(s_y)||_{\mathbb W} ||\delta h^{(n)}_1||_{\mathbb W} \rightarrow 0, 
\end{equation}
as $n \rightarrow \infty$.
It was shown in \cref{sec:Tpositive} that the eigenvalues of $\vec T_n^{-1}$ are bounded from above by $(\inf |S_{yy}|)^{-1} < \infty$, so $\delta \tilde h_{2, k}^{(n)}$ can be estimated as
\begin{equation}\label{eq:WH2normEst}
	\sum_{k = 0}^{n} |\delta h_k^{(n)}| = \sum_{k = 0}^{n} |h_k - h_k^{(n)}| \leq (\inf |S_{yy}|)^{-1} \sum_{l = 0}^{\infty} |\delta S_l| \rightarrow 0. 
\end{equation}
Since
\begin{equation}
	|\delta h_k^{(n)}(u)| \leq \sum_{k = 0}^{n} |\delta h_k^{(n)}| \rightarrow 0
\end{equation}
as $n \rightarrow \infty$, $\delta h_2^{(n)}(u)$ vanishes everywhere and uniformly.

As a result, $\tilde h^{(n)}(u)$ converges to $\tilde h(u)$ uniformly, since $\delta \tilde h_1^{(n)}(u)$ and $\delta \tilde h_2^{(n)}(u)$ vanish uniformly. 
This immediately implies that $h^{(n)}(\omega)$ converges to $h(\omega)$ uniformly due to \cref{lem:TildeMapWiener}. 
As an immediate bonus, $h_p^{(n)}$ converges to $h_p$ uniformly with $n\rightarrow \infty$ for all $p$.
This verifies statements (ii) and (iii) and concludes the proof of \cref{thrm:WienerPartTheorem}.

\subsection{Proof of \cref{thrm:SmoothWienerPartTheorem}}

To prove \cref{thrm:SmoothWienerPartTheorem}, we first establish the equivalence between the functions that satisfy the assumptions of \cref{thrm:SmoothWienerPartTheorem} and the positive non-integer H\"older classes on $\mathbb T$.
The following lemma summarizes the properties of such map.
\begin{lemma}\label{lem:HolderLineToCircle}
Let $g \in L^2(\mathbb R)$ and $\omega_0 > 0$. 
Assume $g$ is of a H\"older class $C^{m, \alpha}$ for integer $m \geq 1$ and $0< \alpha < 1$ on $[- \omega_0, \omega_0]$ and $g^\prime(\omega) = g(1/\omega)$ is of a H\"older class $C^{m, \alpha}$ on $[-\omega^{-1}_0, \omega_0^{-1}]$,
Then and only then $\bar g(u) = g(\tan u/2)$ is of a H\"older class  $C^{m, \alpha}$ on $\mathbb T$.
\end{lemma}
\begin{proof}
Both necessity and sufficiency conditions are proven based on the fact that diffeomorphisms between compact sets preserve the H\"older class.
Indeed, a diffeomorphisms between compact sets is known to preserve the H\"older class $C^{0, \alpha}$ \cite[Proposition 1.2.7]{Fiorenza2017}.
Since the diffeomorphisms between compact spaces also preserve the finiteness of the derivatives of $g$, the H\"older class $C^{m, \alpha}$ is preserved for all non-negative integer $n$.
We proceed to prove the necessity condition first.
We assume that $\bar g$ is a member of $C^{m, \alpha}(\mathbb T)$.
Consider an arbitrary interval $[-u_0, u_0]$ for $u_0< \pi$.
The the map $x = \tan u/2$ is a diffeomorphism between compact sets $[-u_0, u_0] \subset \mathbb T$ and $[-\omega_0, \omega_0] \subset \mathbb R$, where $\omega_0 = \tan u_0/2$.
Therefore, automatically $g \in C^{m, \alpha}([-\omega_0, \omega_0])$, since $\bar g \in C^{m,\alpha}(\mathbb T) \subset C^{m, \alpha}([-u_0, u_0])$
Consider the diffeomorphic map $x = \cot u/2$ between $[u_0, 2 \pi - u_0]$ and $[-\omega_0^{-1}, \omega_0^{-1}]$.
Then since $g(\tan u/2)$ is of a positive H\"older class $C^{n, \alpha}$ on $[u_0, 2 \pi - u_0]$ the map $g^\prime(x) = g(1/x) = \bar g(2 \cot^{-1} x)$ is of a positive H\"older class $C^{m, \alpha}$ on $[-\omega_0^{-1}, \omega_0^{-1}]$.

The proof of the sufficiency condition is constructed similarly.
Consider $\omega_0>0$ and $u_0 = 2 \tan^{-1} \omega_0$.
Repeating the argument above immediately implies that $\bar g \in C^{m, \alpha}([-u_0, u_0])$ and $\bar g \in C^{m,\alpha}([u_0, 2 \pi - u_0])$. 
We proceed to prove that separate H\"older conditions on the intersecting closed intervals imply the H\"older condition in their union.
Let $\bar g$ be such that
\begin{align}
	|\bar g^{(m)}(u_1)- \bar g^{(m)}(u_2)| &\leq C_1 |(u_1 - u_2) \; \mathrm{mod} \; 2 \pi|^\alpha,
	\\
	|\bar g^{(m)}(u_3) - \bar g^{(m)}(u_4)| &\leq C_2 |(u_3 - u_4) \; \mathrm{mod} \; 2 \pi|^\alpha,
\end{align}
for $u_1, u_2 \in [-u_0, u_0]$ and $u_3, u_4 \in [u_0, 2 \pi - u_0]$.
Then to prove the H\"older condition on the whole $\mathbb T$, one only needs to consider the case of $u_1 \in [-u_0, u_0]$ and $u_2 \in [u_0, 2 \pi - u_0]$.
The following estimation for $C = \max(C_1, C_2)$ and $u_3 = \pm u_0$ holds:
\begin{equation}
\begin{split}
    |\bar g^{(m)}(u_1) - \bar g^{(m)}(u_2)| 
    &\leq |\bar g^{(m)}(u_1) - \bar g^{(m)}(u_3)| \\
        &\qquad + |\bar g^{(m)}(u_2) - \bar g^{(m)}(u_3)| \\
    &\leq C (|u_1 - u_3| \; \mathrm{mod} \; 2\pi)^\alpha \\
        &\qquad + C(|u_2 - u_3| \; \mathrm{mod} \; 2\pi)^\alpha \\
    &\leq C (|u_1 - u_3| \; \mathrm{mod} \; 2\pi \\
        &\qquad + |u_2 - u_3| \; \mathrm{mod} \; 2\pi)^\alpha \\
    &= C (|u_1 - u_2| \; \mathrm{mod} \; 2 \pi )^\alpha.
\end{split}
\end{equation}
We used the reverse Minkowski 
inequality $|a|^\alpha + |b|^\alpha \leq (|a| + |b|)^\alpha$ for $0 < \alpha < 1$ to establish the third inequality.
To satisfy the last equality, we select $u_3 = u_0$ when $|u_1 - u_2| \leq \pi$, and $u_3 = - u_0$ otherwise.
Therefore, the H\"older condition holds for all points on $\mathbb T$ and $\bar g \in C^{m,\alpha}(\mathbb T)$. 
\end{proof}

With \cref{lem:HolderLineToCircle}, the following course towards the completion of the proof of \cref{thrm:SmoothWienerPartTheorem} is straightforward. 
According to assumtions of \cref{thrm:SmoothWienerPartTheorem}, $S_{yy}(\omega)$ satisfies the assmptions of \cref{lem:HolderLineToCircle}.
Therefore, $s_y(u) = S_{yy}(\tan u/2)$ is such that $s_y \in C^{m,\alpha}(\mathbb T)$.
Under the $\sim$ map, $S_{xy}(\omega)$ is mapped into $\tilde S_{xy}(u) = \sqrt{\pi}(1 + \ii \omega) S_{xy}(\omega)$, where $\omega = \tan u/2$.
Since $S_{xy}(\omega)$ is a member of $C^{m, \alpha}([-\omega_0, \omega_0])$, then $(1 + \ii \omega) S_{xy}(\omega)$ is also a member of $C^{m,\alpha}([-\omega_0, \omega_0])$.
Therefore, $\tilde S_{xy} \in C^{m, \alpha}(\mathbb T)$, since $\tilde S_{xy}$ satisfies the assumptions of \cref{lem:HolderLineToCircle}.
Note that since $C^{m, \alpha}(\mathbb T) \subset \mathbb W$ for $m \geq 1$ and $\alpha > 0$, the conditions of \cref{thrm:SmoothWienerPartTheorem} are strictly stronger than \cref{thrm:WienerPartTheorem}.

Note that $T(s_y)$ is invertible on $L^2(\mathbb T)$ according to \cref{thrm:ToeplitzInvertable}.
Therefore, $T(s_y)$ is also invertible on $C^{m, \alpha}(\mathbb T) \cap H^\infty$ due to $C^{m, \alpha}(\mathbb T)$ being inverse-closed (\cref{lem:InversionLemma,thrm:HolderInverseClosed}).
Therefore, the filter $\tilde h = T(s_y)^{-1}[\tilde S_{xy}]$ is also a member of $C^{m, \alpha}(\mathbb T)$, 
and therefore satisfies (i) and (ii) of \cref{thrm:SmoothWienerPartTheorem} due to \cref{lem:HolderLineToCircle,lem:CProps}.

To study the convergence of $h^{(n)}(\omega)$ to $h(\omega)$, we perform estimates analogous to the one in \cref{sec:WienerAlgebra}.
The estimate on $\delta \tilde h^{(n)}_1(u)$, defined by \cref{eq:deltah1}, is completed as follows.
Function $\delta \tilde h^{(n)}_1(u)$ is necessarily a member of $C^{m, \alpha}(\mathbb T)$, since it is a difference of $\tilde h(u) \in C^{m, \alpha}(\mathbb T)$ and $\tilde h^{(n)}(u)$, which is a finite Fourier sum.
Therefore, by \cref{lem:CProps},
\begin{equation}\label{eq:Dh1FinEst}
	|\delta \tilde h^{(n)}_{1}(u)| = O\left( \frac{\ln n}{n^{m+\alpha}} \right),
\end{equation}
with the estimate being uniform in $u \in \mathbb T$ and
\begin{equation}\label{eq:CaH1kEst}
	|\delta \tilde h^{(n)}_l| = O \left( \frac{1}{l^{m+\alpha}} \right)
\end{equation}
for $l > m$.

The estimate on $\delta \tilde h^{(n)}_2(u)$ is more cumbersome. 
We begin by estimating $\delta S^{(n)}_l$ defined by \cref{eq:ExactSystem}:
\begin{equation}
	 \sum_{l = 0}^{\infty} |\delta S_l^{(n)}| \leq ||T(s_y)||_{\mathbb W} \sum_{l = 0}^\infty|\delta \tilde h_{1, l}^{(n)}|. 
\end{equation}
Since $s_y \in C^{m, \alpha} \subset \mathbb W$, the Wiener norm of $s_y$ is finite and $||T(s_y)||_{\mathbb W} < \infty$, so
\begin{equation}
	\sum_{l = 0}^{\infty} |\delta S_l^{(n)}| = O\left( \frac{1}{n^{m + \alpha - 1}} \right).
\end{equation}
With the estimation on $\delta S^{(n)}_l$ obtained, we are prepared to estimate 
$\delta \tilde h_2^{(n)}(u)$, which is given by \cref{eq:deltah1}.
From the solution of the truncated equation it follows that
\begin{equation}
	\delta \tilde h^{(n)}_{2, l} = \sum_{k = 0}^{n} [(\vec T^{(n)})^{-1}]_{lk} \delta S_k^{(n)}.
\end{equation}
Applying an estimation analogous to \cref{eq:WH2normEst}, one recovers 
\begin{equation}\label{eq:CaHk2kEst}
	|\delta \tilde h^{(2)}(u)| \leq \sum_{l = 0}^\infty |\delta \tilde h^{(n)}_{2, l}| = O\left( \frac{1}{n^{m+\alpha-1}} \right).
\end{equation}

\cref{eq:CaH1kEst,eq:CaHk2kEst}, when combined, result in
\begin{equation}\label{eq:CaCoefEst}
	|h_p^{(n)} - h_p| = O\left( \frac{1}{n^{m+\alpha - 1}} \right),
\end{equation}
and
\begin{equation}\label{eq:CaFuncEst}
	|\tilde h^{(n)}(u) - \tilde h(u)| = O \left( \frac{1}{n^{m+\alpha - 1}} \right),
\end{equation}
with the estimates being uniform in $p$ and $u$ correspondingly.
When \cref{eq:CaCoefEst,eq:CaFuncEst} are combined with \cref{lem:HolderLineToCircle}, the properties (iii) and (iv) follow and the proof of \cref{thrm:SmoothWienerPartTheorem} is complete.


\section{Supplementary numerical data}\label{sec:ExtraData}
\begin{figure*}[!htbp]
    \centering
    \includegraphics[width=1\linewidth]{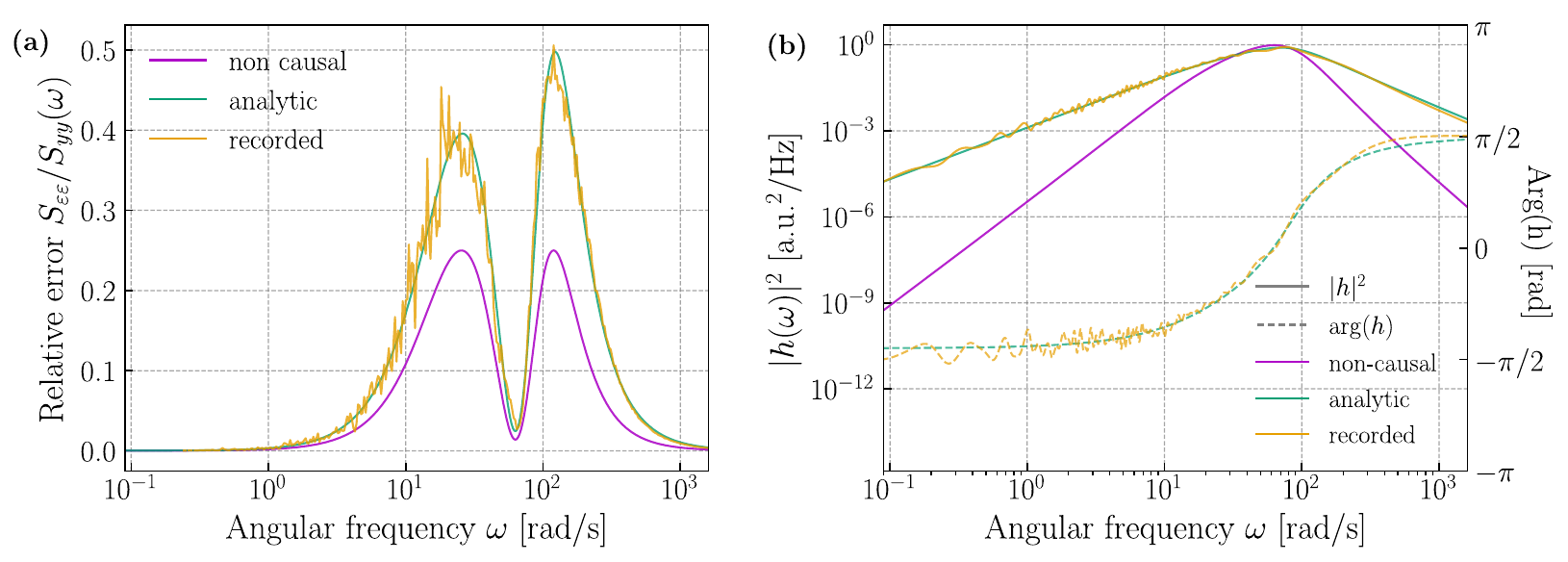}
    \caption{
Figure reproduces Fig. (1.a) and Fig. (1.b). The yellow "recorded" and magenta "non-causal" curves are the same as in Fig. 1. The green "analytic" curve is the Wiener filter obtained directly from analytic forms of the PSD's: $S_{xx}=A \gamma^2/ ((|\omega|-\omega_c)^2+\gamma ^2)$ with parameters $\gamma = 2 \pi$, $A=0.9$ and $\omega_c = 10 \cdot 2\pi$, and $S_{nn}=5/\omega^{1.8}+0.01$. The plot shows how sampling noise impacts the performance of the filter construction procedure by benchmarking it against the solution obtained from smooth PSDs.
}
\end{figure*}

\begin{figure}[!htbp]
    \centering
    \includegraphics[width=1\linewidth]{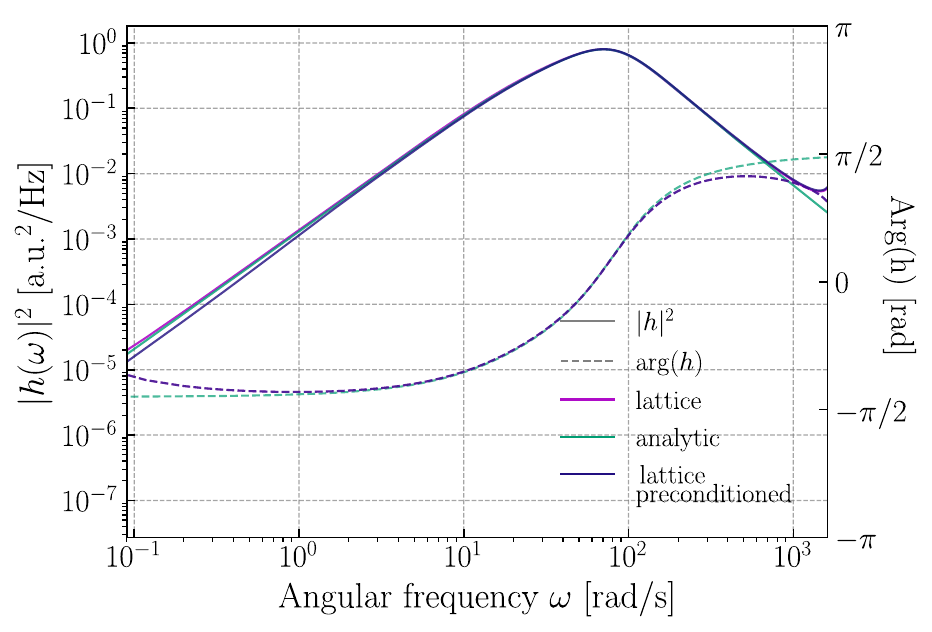}
    \caption{
Figure reproduces Fig. (2.b). The "analytic" curve is the same as the "analytic" curve in Fig. 2. The  "analytic" curve is obtained by preconditioning the smooth data and applying the filter computation procedure.  The "lattice preconditioned" curve was obtained by solving the \cref{Gplus,Gminus} on an equally spaced frequency grid after preconditioning the data.  The "lattice" curve was obtained by solving the \cref{Gplus,Gminus} on the same grid without preconditioning the data. The plot illustrates that the solution obtained with the proposed filter-construction procedure converges to a naive attempt to solve \cref{Gplus,Gminus}. The small difference at small frequencies is attributed to the finite step size of the lattice grid. The small difference at large frequencies is attributed to the growing unregulated logarithmic divergence of the Hilbert transform at the high-frequency edge of the finite frequency grid.
The PSDs used to construct the plot are the same as in Fig 2: $S_{xx}=A \gamma^2/ ((|\omega|-\omega_c)^2+\gamma ^2)$ with parameters $\gamma = 2 \pi$, $A=0.9$ and $\omega_c = 10 \cdot 2\pi$, and $S_{nn}=5/\omega^{1.8}+0.01$. The plot shows how sampling noise impacts the performance of the filter construction procedure by benchmarking it against the solution obtained from smooth PSDs.
}
\end{figure}

\end{document}